\documentclass[oneside,english]{amsart}
\usepackage[T1]{fontenc}
\usepackage[latin9]{inputenc}
\usepackage{textcomp}
\usepackage{amstext}
\usepackage{amsthm}
\usepackage{amssymb}
\usepackage{esint}

\makeatletter
\numberwithin{equation}{section}
\numberwithin{figure}{section}
\theoremstyle{plain}
\newtheorem{thm}{\protect\theoremname}[section]
  \theoremstyle{plain}
  \newtheorem{cor}[thm]{\protect\corollaryname}
  \theoremstyle{remark}
  \newtheorem{rem}[thm]{\protect\remarkname}
  \theoremstyle{plain}
  \newtheorem{prop}[thm]{\protect\propositionname}
  \theoremstyle{plain}
  \newtheorem{lem}[thm]{\protect\lemmaname}
  \theoremstyle{definition}
  \newtheorem{example}[thm]{\protect\examplename}
  \theoremstyle{definition}
  \newtheorem{defn}[thm]{\protect\definitionname}

\def\makebbb#1{
    \expandafter\gdef\csname#1\endcsname{
        \ensuremath{\Bbb{#1}}}
}\makebbb{R}\makebbb{N}\makebbb{Z}\makebbb{C}\makebbb{H}\makebbb{E}\makebbb{H}\makebbb{P}\makebbb{B}\makebbb{Q}\makebbb{E}

\usepackage{babel}

\usepackage{babel}

\makeatother

\usepackage{babel}

\makeatother

\usepackage{babel}
  \providecommand{\corollaryname}{Corollary}
  \providecommand{\definitionname}{Definition}
  \providecommand{\examplename}{Example}
  \providecommand{\lemmaname}{Lemma}
  \providecommand{\propositionname}{Proposition}
  \providecommand{\remarkname}{Remark}
\providecommand{\theoremname}{Theorem}

\begin{document}

\title{From Monge-Ampère equations to envelopes and geodesic rays in the
zero temperature limit}

\author{Robert J. Berman}
\begin{abstract}
Let $(X,\theta)$ be a compact complex manifold $X$ equipped with
a smooth (but not necessarily positive) closed $(1,1)-$form $\theta.$
By a well-known envelope construction this data determines a canonical
$\theta-$psh function $u_{\theta}$ which, in the case when the cohomology
class $[\theta]$ is Kähler, is in the Hölder space $C^{1,\alpha}$
for any $\alpha\in]0,1$ (but, typically, $u_{\theta}$ is not $C^{2}-$smooth).
We introduce a family $u_{\beta}$ of regularizations of $u_{\theta}$,
parametrized by a positive number $\beta,$ where $u_{\beta}$ is
defined as the unique smooth solution of a complex Monge-Ampère equation
of Aubin-Yau type. It is shown that, as $\beta\rightarrow\infty,$
the functions $u_{\beta}$ converge to the envelope $u_{\theta}$
uniformly on $X$ in the strongest possible Hölder sense. A generalization
of this result to the case of a nef and big cohomology class is also
obtained. As a consequence new PDE proofs are obtained for the regularity
results for envelopes in \cite{b-d} (which, however, are weaker than
the results in \cite{b-d} in the case of a non-nef big class). Applications
to the regularization of $\omega-$psh functions and geodesic rays
in the closure of the space of Kähler metrics are given. As briefly
explained there is a statistical mechanical motivation for this regularization
procedure, where $\beta$ appears as the inverse temperature. This
point of view also leads to an interpretation of $u_{\beta}$ as a
``transcendental'' Bergman metric.
\end{abstract}

\maketitle
\tableofcontents{}

\section{Introduction}

Let $X$ be a compact complex manifold equipped with a smooth closed
$(1,1)-$ form $\theta$ on $X$ and denote by $[\theta]$ the corresponding
class in the cohomology group $H^{1,1}(X,\R).$There is a range of
positivity notions for such cohomology classes, generalizing the classical
positivity notions in algebraic geometry. The algebro-geometric situation
concerns the special case when $X$ is projective variety and the
cohomology class in question has integral periods, which equivalently
means that the class may be realized as the first Chern class $c_{1}(L)$
of a line bundle $L$ over $X$ \cite{dem02,dem0,d-p}. Accordingly,
general cohomology classes in $H^{1,1}(X,\R)$ are some times referred
to as\emph{ transcendental} classes and the corresponding notions
of positivity may be formulated in terms of the convex subspace of
positive currents in the cohomology class - the strongest notion of
positivity is that of a\emph{ Kähler class, }which means that the
class contains a Kähler metric, i.e. a smooth positive form (see \cite{d-p}
for equivalent numerical characterizations of positivity). In general,
once the reference element $\theta$ in the cohomology class in question
has been fixed the subspace of positive forms may be identified (mod
$\R)$ with the space $PSH(X,\theta)$ of all \emph{$\theta-$plurisubharmonic
}function $(\theta-$psh, for short), i.e. all integrable upper semi-continuous
functions $u$ on $X$ such that 
\[
\theta+dd^{c}u\geq0,\,\,\,\,\,\,dd^{c}:=i\partial\bar{\partial}
\]
holds in the sense of currents (in the integral case the space $PSH(X,\theta)$
may be identified with the space of all singular positively curved
metrics on the corresponding line bundle $L).$ When the class $[\theta]$
is pseudo-effective, i.e. it contains a positive current, there is
a canonical element in $PSH(X,\theta)$ defined as the following envelope\emph{:}
\[
u_{\theta}(x):=\sup\{u(x):\,\,\,u\leq0,\,\,\,u\in PSH(X,\theta)\},
\]
defining a $\theta-$plurisubharmonic function with minimal singularities
in the sense of Demailly \cite{dem02,begz}. 

In this paper we introduce a natural family of regularizations $u_{\beta}$
of the envelope $u_{\theta},$ indexed by a positive real parameter
$\beta,$ where $u_{\beta}$ is determined by an auxiliary choice
of volume form $dV;$ the functions $u_{\beta}$ will be defined as
solutions to certain complex Monge-Ampère equations, parametrized
by $\beta.$ Several motivations for studying the functions $u_{\beta}$
and their asymptotics as $\beta\rightarrow\infty,$ will be given
below. For the moment we just mention that $u_{\beta}$ can, in a
certain sense, be considered as a ``transcendental'' analog of the
Bergman metric for a high power of a line bundle $L$ over $X$ and
moreover from a statistical mechanical point of view the limit $\beta\rightarrow0$
appears as a zero-temperature limit.

In order to introduce the precise setting and the main results we
start with the simplest case of a\emph{ Kähler class} $[\theta].$
First note that the envelope construction above can be seen as a generalization
of the process of replacing the graph of a given smooth functions
with its convex hull. By this analogy it is already clear from the
one-dimensional case that $u_{\theta}$ will almost never by $C^{2}-$smooth
even if the class $[\theta]$ is Kähler (unless $\theta$ is semi-positive,
so that $u_{\theta}=0).$ However, by the results in \cite{b-d} the
complex Hessian of the function $u_{\theta}$ is always locally bounded
and in particular $u_{\theta}$ is in the Hölder space $\mathcal{C}^{1,\alpha}(X)$
for any $\alpha\in]0,1[$ (see also \cite{berm1} for a slightly more
precise result in the case of a class with integral periods). Fixing
a volume form $dV$ we consider, for $\beta$ a fixed positive number,
the following complex Monge-Ampère equations for a smooth function
$u_{\beta}:$
\begin{equation}
(\theta+dd^{c}u_{\beta})^{n}=e^{\beta u_{\beta}}dV\label{eq:ma eq in intro}
\end{equation}
 By the seminal results of Aubin \cite{au} and Yau \cite{y} there
exists indeed a unique smooth solution $u_{\beta}$ to the previous
equation. In fact, any smooth solution is automatically\emph{ }$\theta-$psh
and the form $\omega_{\beta}:=\theta+dd^{c}u_{\beta}$ defines a Kähler
metric in $[\theta].$ 
\begin{thm}
\label{thm:main thm khler case intro}Let $\theta$ be a smooth $(1,1)-$form
on a compact complex manifold $X$ such that $[\theta]$ is a Kähler
class. Denote by $u_{\theta}$ the corresponding $\theta-$psh envelope
and by $u_{\beta}$ the unique smooth solution of the complex Monge-Ampère
equations \ref{eq:ma eq in intro} determined by $\theta$ and a fixed
volume form $dV$ on $X.$ Then, as $\beta\rightarrow\infty,$ the
functions $u_{\beta}$ converge to $u_{\theta}$ in $\mathcal{C}^{1,\alpha}(X)$
for any $\alpha\in]0,1[,$ with a uniform bound on $dd^{c}u_{\beta}.$
\end{thm}
In particular, the previous theorem yields a new direct PDE proof
of the Laplacian bound on $u_{\theta}$ in \cite{b-d} in the case
of a Kähler class, with a rather explicit geometrical control on the
bound. More generally, the proof reveals that the result remains valid
if $dV$ is replaced by any family $dV_{\beta}$ of volume forms such
that $dd^{c}(\log(dV_{\beta}/dV_{1})=o(\beta).$ As a consequence
the convergence result above admits the following geometric formulation:
let $\omega_{\beta}$ be a family of Kähler metrics in $[\theta]$
satisfying the following twisted Kähler-Einstein equation: 
\[
\mbox{Ric \ensuremath{\omega_{\beta}=-\beta\omega_{\beta}+\beta\theta+o(\beta),}}
\]
where $\mbox{Ric \ensuremath{\omega_{\beta}}}$ denotes the form representing
the Ricci curvature of the Kähler metric $\omega_{\beta}$ and $o(\beta)$
denotes a family of forms on $X$ such that $o(\beta)/\beta\rightarrow0$
in the $L^{\infty}-$sense as $\beta\rightarrow\infty.$ Then the
previous theorem says that $\omega_{\beta}$ is uniformly bounded
and converges to $\theta+dd^{c}u_{\theta}$ in the sense of currents
and the normalized potentials of $\omega_{\beta}$ converge in $\mathcal{C}^{1,\alpha}(X)$
to $u_{\theta}.$

More generally, we will consider the case when the cohomology class
$[\theta]$ is merely assumed to be \emph{big; }this is the most general
setting where complex Monge-Ampère equations of the form make sense
\cite{begz}. The main new feature in this general setting is the
presence of $-\infty-$ singularities of all $\theta-$psh functions
on $X.$ Such singularities are, in general, inevitable for cohomological
reasons. Still, by the results in \cite{begz}, the corresponding
complex Monge-Ampère equations admit a unique $\theta-$psh function
$u_{\beta}$ with\emph{ }minimal singularities\emph{; }in particular
its singularities can only appear along a certain complex subvariety
of $X,$ determined by the class $[\theta]$, whose complement is
called the\emph{ Kähler locus} $\Omega$ of $[\theta]$ (or the\emph{
ample locus)} introduced in \cite{bo} (which in the algebro-geometric
setting corresponds to the complement of the \emph{augmented base
locus} of the corresponding line bundle). Moreover, in the case when
the class $[\theta]$ is also assumed to be\emph{ nef }the solution
$u_{\beta}$ is known to be smooth on $\Omega,$ as follows from the
results in \cite{begz}. In this general setting our main result may
be formulated as follows:
\begin{thm}
\label{thm:main thm big case intro}Let $\theta$ be a smooth $(1,1)-$form
on a compact complex manifold $X$ such that $[\theta]$ is a big
class. Then, as $\beta\rightarrow\infty,$ the functions $u_{\beta}$
converge to $u_{\theta}$ uniformly, in the sense that $\left\Vert u_{\beta}-u_{\theta}\right\Vert _{L^{\infty}(X)}\rightarrow0.$
Moreover, if the class $[\theta]$ is also assumed to be nef, then
the convergence holds in $\mathcal{C}_{loc}^{1,\alpha}(\Omega)$ on
the Kähler locus $\Omega$ of $X.$ 
\end{thm}
In particular, in the general setting of a big class the proof of
the previous theorem yields a new proof of a weaker form of the regularity
result in \cite{b-d} saying that 
\begin{equation}
(\theta+dd^{c}u_{\theta})^{n}\leq1_{D}\theta^{n},\,\,\,\,D=\{x\in X:\,u_{\theta}(x)=0\}\label{eq:bound on ma intro}
\end{equation}
Even though this bound is considerably weaker than the general regularity
result in \cite{b-d} it appears to be adequate for all current complex
geometric applications of envelopes as above, such as the recent proof
of the duality between the pseudoeffective and the movable cone on
a projective manifold in \cite{WN}. 

Some further remarks are in order. First of all, as pointed out above,
it was previously known that the norm $\left\Vert u_{\beta}-u_{\theta}\right\Vert _{L^{\infty}(X)}$
is finite for any fixed $\beta$ (since $u_{\beta}$ and the envelope
$u_{\theta}$ both have minimal singularities) and the thrust of the
first statement in the previous theorem is thus that the norm in fact
tends to zero. This global uniform convergence  is considerably stronger
than a a local uniform convergence on $\Omega.$ Secondly, it should
be stressed that, as shown in \cite{b-d}, the complex Hessian of
the envelope $u_{\theta}$ is locally bounded on $\Omega$ for\emph{
any} big class $[\theta]$ and hence it seems natural to expect that
the local convergence on $\Omega$ in the previous theorem always
holds in the $\mathcal{C}_{loc}^{1,\alpha}(\Omega)-$topology, regardless
of the nef assumption. However, already the smoothness on $\Omega$
of solutions of complex Monge-Ampère equations of the form \ref{eq:ma eq in intro}
is an open problem; in fact, it even seems to be unknown whether there
always exists a $\theta-$psh functions with minimal singularities,
which is smooth on $\Omega.$ On the other hand, for special big classes
$[\theta]$, namely those which admit an appropriate Zariski decomposition
on some resolution of $X,$ the regularity and convergence problem
can be reduced to the nef case (in the line bundle case this situation
appears if the corresponding section ring is finitely generated).

\subsection{Degenerations induced by a divisor and applications to geodesic rays}

In the case of a Kähler class and when $\theta$ is positive, i.e.
$\theta$ is Kähler form, it follows immediately from the definition
that $u_{\theta}=0$ and in this case the convergence in Theorem \ref{thm:main thm khler case intro}
holds in the $\mathcal{C}^{\infty}-$sense, as recently shown in \cite{jmr}
using a completely different proof. However, as shown in \cite{r-w3,r-w1}
in the integral case $[\omega]=c_{1\text{ }}(L),$ a non-trivial variant
of the previous envelopes naturally appear in the geometric context
of\emph{ test configurations} for the polarized manifold $(X,L),$
i.e. $\C^{*}-$equivariant polarized deformations $(\mathcal{X},\mathcal{L})$
of $(X,L)$ and they can be used to construct (weak) \emph{geodesic
rays} in the space of all Kähler metrics in $[\omega].$ Such test
configurations were introduced by Donaldson in his algebro-geometric
definition of K-stability of a polarized manifold $(X,L),$ which
according to the the Yau-Tian-Donaldson is equivalent to the existence
of a Kähler metric in the class $c_{1}(L)$ with constant scalar curvature.
Briefly, K-stability of $(X,L)$ amounts to the positivity of the
Donaldson-Futaki invariants for all test configurations, which in
turn is closely related to the large time asymptotics of Mabuchi's
K-energy functional along the corresponding geodesic rays (see \cite{p-s2}
and references therein). 

Let us briefly explain how this fits into the present setup in the
special case of the test configurations defined by the deformation
to the normal cone of a divisor $Z$ in $X$ (e.g. a smooth complex
hypersurface in $X).$ First we consider the following complex Monge-Ampère
equations degenerating along the divisor $Z,$ 
\[
(\omega-\lambda\theta_{L}+dd^{c}u)^{n}=e^{\beta u}\left\Vert s\right\Vert ^{2\lambda\beta}dV,
\]
 where we have realized $Z$ as the zero-locus of a holomorphic section
$s$ of a line bundle $L$ over $X$ equipped with a fixed Hermitian
metric $\left\Vert \cdot\right\Vert $ with curvature form $\theta_{L}$
and where $\lambda\in[0,\infty[$ is an additional fixed parameter.
As is well-known, for $\lambda$ sufficiently small $(\lambda\leq\epsilon)$
there is, for any $\beta>0,$ a unique continuous $\omega-\lambda\theta_{L}-$psh
solution $u_{\beta,\lambda}$ to the previous equation, which is smooth
on $X-Z.$ We will show that, when $\beta\rightarrow\infty,$ the
solutions $u_{\beta,\lambda}$ converge in $\mathcal{C}^{1,\alpha}(X)$
to a variant of the envelope $u_{\theta},$ that we will (abusing
notation slightly) denote by $u_{\lambda}:$ 
\[
u_{\lambda}(x):=\sup\{u(x):\,\,\,u\leq-\lambda\log\left\Vert s\right\Vert ^{2}\,\,\,u\in PSH(X,\omega-\lambda\theta_{L})\}
\]
(see section \ref{sub:Degenerations-induced-by}). It may identified
with the envelopes with prescribed singularities introduced in \cite{berm1}
in the context of Bergman kernel asymptotics for holomorphic sections
vanishing to high order along a given divisor (see \cite{r-w1} for
detailed regularity results for such envelopes and the relations to
Hele-Shaw type flows and \cite{p-s} for related asymptotic results
in the toric case). 

Remarkably, as shown in \cite{r-w3,r-w1} (in the line bundle case)
taking the Legendre transform of the envelopes $u_{\lambda}+\lambda\log\left\Vert s\right\Vert ^{2}$
with respect to $\lambda$ produces a geodesic ray in the closure
of the space of Kähler potentials in $[\omega],$ which coincides
with the $C^{1,\alpha}-$geodesic constructed by Phong-Sturm \cite{p-s1b,p-s3}
(in general, the geodesics are not $C^{2}-$smooth). Here, building
on \cite{r-w3,r-w1}, we show that the logarithm of the Laplace transform,
with respect to $\lambda,$ of the Monge-Ampère measures of the envelopes
$u_{\lambda}$ defines a family of subgeodesics in the space of Kähler
potentials converging to the corresponding geodesic ray (see Cor \ref{cor:geod ray in text}).
In geometric terms the result may be formulated as follows
\begin{cor}
\label{cor:geod intro}Let $\omega$ be a Kähler form, and fix a constant
$c$ such that $[\omega]-c[Z]$ is a Kähler class. Let $\omega_{\beta,\lambda}$
be a family of currents in $[\omega]-\lambda[Z],$ defining smooth
Kähler metrics away from the support of $Z$ and satisfying 
\[
\mbox{Ric \ensuremath{\omega_{\beta,\lambda}=-\beta\omega_{\beta,\lambda}+\beta(\omega-\lambda[Z])+o(\beta)}}
\]
Then 
\[
\varphi_{\beta}^{t}:=\frac{1}{\beta}\log\int_{[0,c]}d\lambda e^{\beta(\lambda-c)t}\frac{\omega_{\beta,\lambda}^{n}}{\omega^{n}}
\]
defines a family of subgeodesics converging in $C^{0}(X\times[0,T]),$
for any fixed $T>0,$ to a geodesic ray $\varphi^{t}$ associated
to the test configuration $(\mathcal{X},\mathcal{L}_{c})$ defined
by the deformation to the normal cone of $Z.$ Moreover, in the case
when $[\omega]\in H^{2}(X,\Q)$ the convergence holds in $C^{0}(X\times[0,\infty].$
\end{cor}
This can be seen as a ``transcendental'' analogue of the approximation
result of Phong-Sturm \cite{p-s-4}, which uses Bergman geodesic rays.
However, while the latter convergence result holds point-wise almost
everywhere and for $t$ fixed, an important feature of the convergence
in the previous corollary is that it is \emph{uniform,} even when
$t$ ranges in all of $[0,\infty[.$ More generally, we will establish
an extension of the previous result to the case when $[\omega]-c[Z]$
(or equivalently $\mathcal{L}_{c})$ is merely assumed big.

The motivation for considering this ``transcendental'' approximation
scheme for geodesic rays is two-fold. First, as is well-known, recent
examples indicate that a more ``transcendental'' notion of K-stability
is needed for the validity of the Yau-Tian-Donaldson conjecture, obtained
by relaxing the notion of a test configuration. One such notion, called
\emph{analytic test configurations,} was introduced in \cite{r-w3}
and as shown in op. cit. any such test configuration determines a
weak geodesic ray, which a priori has very low regularity. However,
the approximation scheme above could be used to regularize the latter
weak geodesic rays, which opens the door for defining a notion of
generalized Donaldson-Futaki invariant by studying the large time
asymptotics of the K-energy functional along the corresponding regularizations
(as in the Bergman metrics approach in \cite{p-s-4}). In another
direction, the approximation scheme above should be useful when considering
the analog of K-stability for a non-integral Kähler class $[\omega]$
(compare section \ref{sec:Applications-to-geodesic}). The previous
corollary is just a first illustration of this approximation scheme
and we leave the development of more general approximation results
for the future.

\subsubsection*{On the proofs}

Next, let us briefly discuss the proofs of the previous theorems,
starting with the case of a Kähler class. First, the weak convergence
of $u_{\beta}$ towards $u_{\beta}$ (i.e. convergence in $L^{1}(X))$
is proved using variational arguments (building on \cite{bbgz}).
In fact, we will give two different proofs of this convergence, where
the first one is variational and has two merits: $(i)$ it generalizes
directly to the case of a big class and $(ii)$ it applies when $dV$
is replaced with a quite singular measure $\mu_{0}$ (satifying a
Bernstein-Markov property). The second proof uses a direct simple
maximum principle argument.

In either way, to conclude the proof of Theorem \ref{thm:main thm khler case intro}
we just have to provide a priori estimates on $u_{\beta},$ which
are uniform in $\beta$ and which we deduce from Siu's variant of
the Aubin-Yau Laplacian estimates. In particular, this implies convergence
in $L^{\infty}(X).$ However, in the case of a general big class,
in order to establish the global $L^{\infty}-$convergence, we need
to take full advantage of the variational argument, namely that the
argument shows that $u_{\beta}$ converges to $u_{\theta}$\emph{
in energy }and not only in $L^{1}(X).$ This allows us to invoke the
$L^{\infty}-$stability results in \cite{g-z2}. Briefly, the point
is that convergence in energy implies convergence in capacity, which
together with an $L^{p}-$control on the corresponding Monge-Ampère
measures opens the door for Kolodziej type $L^{\infty}-$estimates.
Moreover, a variant of the maximum principle argument used in the
case of the Kähler class, based on the theory of viscosity subsolutions
developed in \cite{e-g-z}, yields the bound \ref{eq:bound on ma intro}
(only the local case of the results in \cite{e-g-z} is needed).

In particular, in the Kähler case we obtain a new simple PDE proof
of the regularity result for $u_{\theta}$ in \cite{b-d}, based on
a priori estimates, which should be contrasted with the proof in \cite{b-d},
which uses completely different pluripotential theoretic arguments.
These latter argument involve Demailly's deep extension of the Kiselman
technique for attenuating singularities (compare \cite{dem0})  and
they have the virtue of applying in the general setting of a big class.
Conversely, it would be very interesting if a similar pluripotential
theoretic argument could be used to establish the conjectural smoothness
of $u_{\beta}$ on the Kähler locus $\Omega,$ thus avoiding the difficulties
which appear when trying to use a priori estimates in the setting
of a big class. There are certainly strong indications that this can
be done (see for example Remark \ref{rem: apriori}), but we shall
leave this problem for the future.

\subsection{\label{sub:Further-background-and}Further background and motivation}

Before turning to the proofs of the results introduced above it may
be illuminating to place the result into a geometric and probabilistic
context (see also Section \ref{sub:Transcendental-Bergman-kernels}
for the relation to Bergman kernel asymptotics).

\subsubsection*{Kähler-Einstein metrics and the continuity method}

First of all we recall that the main geometric motivation for studying
complex Monge-Ampère equations of the form \ref{eq:ma eq in intro}
comes from \emph{Kähler-Einstein geometry} and goes back to the seminal
works of Aubin \cite{au} and Yau \cite{y} in setting when $X$ is
a canonically polarized projective algebraic variety, i.e. the canonical
line bundle $K_{X}:=\Lambda^{n}T^{*}X$ of $X$ is ample. If the form
$\theta$ is taken as a Kähler metric $\omega$ on $X$ in the first
Chern class $c_{1}(K_{X})$ of $K_{X}$ and $dV$ is chosen to be
depend on $\omega$ in a suitable sense (i.e. $dV=e^{h_{\omega}}\omega^{n},$
where $h_{\omega}$ is the Ricci potential of $\omega)$, then the
corresponding solution $u_{\beta}$ of the equation \ref{eq:ma eq in intro}
for $\beta=1$ is the Kähler potential of a \emph{Kähler-Einstein
metric} $\omega_{KE}$ on $X$ with negative Ricci curvature. Similarly,
in the case of $\beta=-1$ the equation \ref{eq:ma eq in intro} corresponds
to the Kähler-Einstein equation for a \emph{positively} curved Kähler-Einstein
equation in $c_{1}(-K_{X})$ on a Fano manifold $X.$ For a general
value on the parameter $\beta$ the equation appears in the continuity
method for the Kähler-Einstein equation. Indeed, for $L=\pm K_{X}$
the equation \ref{eq:ma eq in intro} is equivalent to the following
equation for $\omega_{\beta}$ in $c_{1}(L)$

\begin{equation}
\mbox{Ric \ensuremath{\omega_{\beta}=-\beta\omega_{\beta}+(\beta-\pm1)\theta,}}\label{eq:aubins equation}
\end{equation}
which, for $\beta$ negative, is precisely Aubin's continuity equation
for the Kähler-Einstein problem on a Fano manifold (when $\theta$
is taken as Kähler form in $c_{1}(\pm K_{X})).$ In the present setting,
where $c_{1}(\pm K_{X})$ is replaced by a general Kähler (or big)
cohomology class $[\theta]$ there is no canonical volume form $dV$
attached to $\theta$ and we thus need to work with a general volume
form $dV,$ but this only changes the previous equation with a term
which is independent of $\beta$ and which, as we show, becomes negligible
as $\beta\rightarrow\infty.$ 

Interestingly, as observed in \cite{ru} the equation \ref{eq:aubins equation}
can also be obtained from the \emph{Ricci flow} via a backwards Euler
discretization. Accordingly, the corresponding continuity path is
called the Ricci continuity path in the recent paper \cite{jmr},
where it (or rather its ``conical'' generalization) plays a crucial
role in the construction of Kähler-Einstein metrics with edge/cone
singularities, by deforming the ``trivial'' solution $\omega_{\beta}=\theta$
at $\beta=\infty$ to a Kähler-Einstein metric at $\beta=\pm1.$ It
should however be stressed that the main point of the present paper
is to study the case of a \emph{non-positive} form $\theta$ which
is thus different from the usual settings appearing in the context
of Kähler-Einstein geometry and where, as we show, the limit as $\beta\rightarrow\infty$
is a canonical positive current associated to $\theta.$

\subsubsection*{Cooling down: the zero temperature limit}

In \cite{berm5,berm8} a probabilistic approach to the construction
of Kähler-Einstein metrics, was introduced, using certain $\beta-$deformations
of determinantal point processes on $X$ (which may be described in
terms of ``free fermions'' \cite{berm5}). The point is that if
$\theta$ is the curvature form of a given Hermitian metric $\left\Vert \cdot\right\Vert $
on a, say ample, line bundle $L\rightarrow X,$ then 
\begin{equation}
\mu^{(N_{k},\beta)}:=\frac{\left\Vert (\det S^{(k)})(x_{1},x_{2},...x_{N_{k}})\right\Vert ^{2\beta/k}dV^{\otimes N_{k}}}{Z_{k,\beta}}\label{eq:prob measure}
\end{equation}
defines a random point process on $X,$ i.e. symmetric probability
measure on the space $X^{N_{k}}$ (modulo the permutation group) of
configurations of $N_{k}$ points on $X,$ where $N_{k}$ is dimension
of the vector space $H^{0}(X,L^{\otimes k})$ of global holomorphic
sections of $L^{\otimes k}$ and $\det S^{(k)}$ is any fixed generator
in the top exterior power $\Lambda^{N_{k}}H^{0}(X,L^{\otimes k}),$
identified with a holomorphic section of $(L^{\otimes k})^{\boxtimes N_{k}}\rightarrow X^{N_{k}}.$ 

From a statistical mechanical point of view the parameter $\beta$
appears as the ``thermodynamical $\beta"$, i.e. $\beta=1/T$ is
the \emph{inverse temperature} of the underlying statistical mechanical
system and the complex Monge-Ampère equations above appear as the
mean field type equations describing the macroscopic equilibrium state
of the system at inverse temperature $\beta.$ More precisely $\mu_{\beta}:=MA(u_{\beta}$)
describes the expected macroscopic distribution of a single particle
when $k$ and (hence also the number of particles $N_{k})$ tends
to infinity, 
\[
\int_{X^{N_{k-1}}}\mu^{(N_{k},\beta)}\rightarrow\mu_{\beta}
\]
 A formal proof of this convergence was first outlined in \cite{berm5}
and then a rigorous proof was obtained in \cite{berm8} (in fact,
a much stronger convergence result holds, saying that the convergence
towards $\mu_{\beta}$ holds exponentially in probability in the sense
of large deviations with a rate functional which may be identified
with the twisted K-energy functional). Anyway, here we only want to
provide a statistical motivation for the large $\beta-$limit, which
thus corresponds to the zero-temperature limit, where the system is
slowly cooled down. From this point of view the convergence result
in Theorem \ref{thm:main thm khler case intro} can then be interpreted
as a \emph{second order phase transition} for the corresponding equilibrium
measures $\mu_{\beta}.$ Briefly, the point is that while the support
of $\mu_{\beta}$ is equal to all of $X$ for any\emph{ finite} $\beta$
the limiting measure $\mu_{\infty}(=MA(u_{\theta}))$ is supported
on a proper subset $S$ of $X$ as soon as $\theta$ is not globally
positive. The formation of a limiting ordered structure (here $MA(u_{\theta})$
and its support $S)$ in the zero-temperature limit is typical for
second order phase transitions in the study of disordered systems.
In fact, in many concrete examples the limiting support $S$ is a
domain with piece-wise smooth boundary, but it should be stressed
that there are almost no general regularity results for the boundary
of $S$ (when $n>1).$ In the one-dimensional case of the Riemann
sphere the support set $S$ appears as the ``droplet'' familiar
from the study of Coulomb gases and normal random matrices (see \cite{za,h-m}
and references therein).

\subsubsection*{Added in proof}

Since the first preprint version of the present paper appeared on
ArXiv there has been a number of interesting developments that we
briefly describe. In \cite{d-r} it was shown that $u_{\theta}$ is
Lipschitz continuous as soon as $\theta$ has a Lipschitz potential,
using the regularizations $u_{\beta}$ above and Blocki's gradient
estimate (as a replacement of the Aubin-Yau-Siu inequality used in
Prop \ref{prop:laplace estimate in khler case}). Moreover, very recently
the convergence result for $u_{\beta}$ in the present paper was used
to prove the $C^{1,1}-$regularity of $u_{\theta}$ (in the case of
a Kähler class), by using the recent $C^{1,1}-$estimates in \cite{c-t-w}
as a replacement of the Aubin-Yau-Siu inequality. In another direction
it was shown in \cite{c-n} how to extend the $C^{0}-$convergence
implicit in Theorem \ref{thm:main thm khler case intro} to the setting
of Hessian equations on Kähler manifolds, leading to a new global
regularization result for $(\omega,m)-$subharmonic functions (see
Remark \ref{rem:chinh }). Furthermore, very recently it was shown
in \cite{sj} and \cite{de-ro}, independently, that a transcendtal
Kähler class containing a constant scalar curvature metric is K-semistable,
in general, and K-stable \cite{de-ro} if the automorphism group is
discrete, which thus establishes one direction of the generalized
Yau-Tian-Donaldson conjecture discussed in Section \ref{sub:A-generalization-of}.
Finally, solutions $u_{\beta}$ of global complex Monge-Ampère equations
as above and their relative positivity properties were used in \cite{clp}
to give an alternative proof of Chen's conjecture concerning the convexity
of the K-energy (recently established in\cite{b-bern2}) with $u_{\beta}$
replacing the local Bergman metric approximations used in \cite{b-bern2},
which thus reinforces the intepretation of $u_{\beta}$ as a transcendtal
Bergman metric discussed in Section \ref{sub:Transcendental-Bergman-kernels}.
Finally, a dynamical analog of Theorem \ref{thm:main thm khler case intro},
formulated in terms of the zero-temperature limit of the twisted Kähler-Ricci
flow, is obtained in \cite{b-l}.

\subsubsection*{Acknowledgements}

It is a pleasure to thanks David Witt-Nyström for illuminating discussions
on the works \cite{witt,r-w3}, Julius Ross for inspiring discussions
on the Yau-Tian-Donaldson conjecture for transcendantal Kähler classes,
Jean-Pierre Demailly for the stimulating colaboration \cite{b-d},
which is one of the motivations for the current work and Chinh Lu
and Yanir Rubinstein for helpful comments on the first preprint version
of the present paper. This work was partly supported by grants from
the European Research Council and Knut and Alice Wallenberg foundation.

\subsubsection{Organization. }

After having setup the general framework in Section \ref{sec:From-Monge-Amp=0000E8re-equations}
we go on to first prove the main result (Theorem \ref{thm:main thm khler case intro})
in the case of Kähler class (by two different proofs) and then its
generalization to big classes (Theorem \ref{thm:main thm big case intro}).
The interpretation in terms of transcendental Bergman metrics is discussed
in Section \ref{sub:Applications-to-regularization}, together with
applications to regularization of $\omega-$psh functions. Then in
Section \ref{sub:Degenerations-induced-by} we consider the singular
version of the previous setup which appears in the presence of a divisor
$Z$ on $X.$ Finally, the results in the latter section are applied
in Section \ref{sec:Applications-to-geodesic} to the construction
and regularization of geodesic rays and relations to the transcendtal
generalization of the Yau-Tian-Donaldson conjecture are discussed.

\section{\label{sec:From-Monge-Amp=0000E8re-equations}From Monge-Ampère equations
to $\theta-$psh envelopes}

Let $X$ be a compact complex manifold equipped with a smooth closed
$(1,1)-$ form $\theta$ and denote by $[\theta]$ the corresponding
(Bott-Chern) cohomology class of currents: 
\[
[\theta]:=\left\{ \theta+dd^{c}u:\,\,\,\,u\in L^{1}(X)\right\} \,\,\,\,\,(dd^{c}:=\frac{i}{2\pi}\partial\bar{\partial})
\]
The space of all $\theta-$plurisubharmonic functions, denoted by
$PSH(X,\theta),$ is the convex subspace of $[\theta]$ consisting
of all $u\in L^{1}(X)$ which are upper semi-continuous (usc) and
such that $\theta+dd^{c}u\geq0$ in the sense of currents. We equip,
as usual, the space $PSH(X,\theta)$ with its $L^{1}-$topology. The
class $[\theta]$ is said to be\emph{ pseudo-effective} if $PSH(X,\theta)$
is non-empty. There is then a canonical element $u_{\theta}$ in the
space $PSH(X,\theta)$ defined as the following envelope: 
\begin{equation}
u_{\theta}(x):=\sup\{u(x):\,\,\,u\leq0,\,\,\,u\in PSH(X,\theta)\},\label{eq:def of env in text}
\end{equation}
Given a smooth function $u$ we will write 
\[
MA_{\theta}(u):=(\theta+dd^{c}u)^{n}
\]
 for the corresponding Monge-Ampère operator (often dropping the subindex
$\theta$ from the notation). In the case when the class $[\theta]$
is a \emph{Kähler class,} i.e. $[\theta]$ contains a smooth and strictly
positive form $\omega$ (i.e. a Kähler form) we will, also fixing
volume form $dV$ on $X,$ denote by $u_{\beta}$ the unique solution
to the following complex Monge-Ampère equation: 
\begin{equation}
MA(u_{\beta})=e^{\beta u_{\beta}}dV\label{eq:ma eq in complex setup}
\end{equation}
(the solution is automatically $\theta-$psh). More generally, the
previous equation makes sense as long as the class $[\theta]$ is\emph{
big} (see section \ref{sub:The-case-of big class} below), but in
general the unique $\theta-$psh solution $u_{\beta}$ will have $-\infty-$singularities
(even if the singularities are always minimal \cite{begz}). We recall
the following regularity result:
\begin{thm}
\cite{b-d}\label{thm:.-reg thm b-d}. Let $\theta$ be a smooth $(1,1)-$form
on a compact complex manifold $X$ such that $[\theta]$ defines a
big cohomology class. Then the Laplacian of $u_{\theta}$ is locally
bounded on a Zariski open subset $\Omega$ of $X$ (which can be taken
as the Kähler locus of $[\theta]).$ As a consequence,$MA(u_{\theta})$
has an $L^{\infty}-$density, or more precisely: 
\begin{equation}
(\theta+dd^{c}u_{\theta})^{n}=1_{D}\theta^{n},\label{eq:point wise formula for MA of env in text}
\end{equation}
 where $D:=\{u_{\theta}=0\}.$
\end{thm}
Note that it follows immediately from the previous theorem that $MA(u_{\theta})$
is supported on the open set $\{u_{\theta}<0\},$ i.e. the following
``orthogonality relation'' holds 
\begin{equation}
\int_{X}u_{\theta}MA(u_{\theta})=0\label{eq:orthog relation}
\end{equation}
(which can be proved directly, only using that $\theta$ has lower
semi-continuous potentials, using well-known properties of free envelopes
which are proved by solving the local Dirichlet problem for complex
Monge-Ampère operator on a ball \cite{b-b}). In the present paper
we will obtain a direct PDE proof of the previous regularity theorem
in the case of a class which is nef and big. For a general big class
our approach will only yield the weaker regularity property 
\begin{equation}
(\theta+dd^{c}u_{\theta})^{n}\leq1_{D}\theta^{n}\label{eq:weaker reg prop}
\end{equation}

\subsubsection{\label{sub:An-alternative-formulation}An alternative formulation
in the Kähler case}

It may be worth pointing out that, in the Kähler case, the following
equivalent formulation of the previous setup may be given, where the
role of smooth form $\theta$ is played by a smooth function $f.$
We start by fixing a Kähler form $\omega$ on $X$ and consider the
corresponding Kähler class $[\omega].$ We can then define a projection
operator $P_{\omega}$ from $\mathcal{C}^{\infty}(X)$ to $PSH(X,\omega)$
by setting 
\begin{equation}
(P_{\omega}f)(x):=\sup\{\varphi(x):\,\,\,\varphi\leq f,\,\,\,\varphi\in PSH(X,\omega)\}\label{eq:def of proj operator in khler case}
\end{equation}
Setting $\theta:=\omega+dd^{c}f$ we see that $u_{\theta}=P_{\omega}f-f.$
Similarly, given a volume form $dV$ on $X$ we denote by $\varphi_{\beta}(:=P_{\beta}(f))$
the unique smooth solution to 
\begin{equation}
(\omega+dd^{c}\varphi_{\beta})^{n}=e^{\beta(\varphi_{\beta}-f)}dV\label{eq:ma equation with k=0000E4hler ref}
\end{equation}
so that $u_{\beta}=\varphi_{\beta}-f.$ One advantage of this new
formulation is that it allows one to consider case where $f$ is allowed
to have $+\infty-$singularities, leading to degeneracies in the rhs
of the previous Monge-Ampère equation. In particular, this will allow
us to consider a framework of complex Monge-Ampère equations degenerating
along a fixed divisor $Z$ in $X.$ Interestingly, this latter framework
can, from the analytic point view, be seen as a variant of the setting
of a big class within a Kähler framework.

We will be interested in the limit when $\beta\rightarrow\infty.$
In order to separate the different kind of analytical difficulties
which appear in the case when $[\theta]$ is Kähler from those which
appear in the general case when $[\theta]$ is big, we will start
with the Kähler case, even though it can be seen as a special case
of the latter.

\subsection{\label{sub:The-case-of}The case of a Kähler class (Proof of Theorem
\ref{thm:main thm khler case intro})}

In this section we will assume that $[\theta]$ is a Kähler class,
i.e. there exists some smooth function $v\in PSH(X,\theta)$ such
that $\omega:=\theta+dd^{c}v>0,$ i.e. $\omega$ is a Kähler form.

\subsubsection{Convergence in energy}

For a given smooth function $u$ we will write 
\begin{equation}
\mathcal{E}(u):=\frac{1}{n+1}\int_{X}\sum_{j=0}^{n}u(\theta+dd^{c}u)^{j}\wedge\theta^{n-j}\label{eq:def of energy functional in khler case}
\end{equation}
More generally, the functional $\mathcal{E}(u)$ extends uniquely
to the space $PSH(X,\theta),$ by demanding that it be increasing
and (strongly) usc \cite{bbgz}. Following \cite{bbgz} we will say
that a sequence $u_{j}$ in $PSH(X,\theta)$ converges to $u$ \emph{in
energy} if $u_{j}\rightarrow u$ in $L^{1}(X)$ and $\mathcal{E}(u_{j})\rightarrow\mathcal{E}(u).$ 

We recall that the functional $\mathcal{E}$ restricted to the convex
space $PSH(X,\theta)\cap L^{\infty}(X)$ (or more generally, to the
finite energy space $\{\mathcal{E}^{1}>-\infty\})$ may be equivalently
defined as a primitive for the Monge-Ampère operator, viewed as a
one-form on the latter space, in the sense that
\begin{equation}
d\mathcal{E}_{|u}=MA(u)\label{eq:e as primitive of ma}
\end{equation}
(i.e. $d\mathcal{E}(u+tv)/dt=\int MA(u)v$ at $t=0).$ 

The next theorem concerns the following general setting: given a finite
measure $\mu_{0}$ on $X$ we denote by $u_{\beta}$ the solution
to the equations \ref{eq:ma eq in complex setup} obtained by replacing
$dV$ with $\mu_{0}$ (the existence of a solution with full Monge-Ampère
mass is equivalent to $\mu_{0}$ not charging pluripolar subsets of
$X).$ Following \cite{b-b-w} the measure $\mu_{0}$ is said to have
the \emph{Bernstein-Markov property wrt $PSH(X,\theta)$} if for any
positive constant $\epsilon$ there exists a constant $C$ such that
for any $u\in PSH(X,\theta)$ 
\begin{equation}
\sup_{X}e^{\beta u}\leq e^{C}e^{\epsilon\beta}\int_{X}e^{\beta u}\mu_{0}\label{eq:BM property}
\end{equation}
 In particular, any volume form $dV$ has the Bernstein-Markov property
wrt $PSH(X,\theta)$ (as follows from the local submean property of
psh functions). 
\begin{thm}
\label{thm:conv in energy khler case}Let $\mu_{0}$ be a finite measure
on $X$ not charging pluripolar subsets. Denote by $u_{\beta}$ the
solution to the complex Monge-Ampère equation determined by the data
$(\theta,\mu_{0},\beta).$ If $\mu_{0}$ has the Bernstein-Markov
property wrt $PSH(X,\theta),$ then $u_{\beta}$ converges to $u_{\theta}$
in energy.\end{thm}
\begin{proof}
Without loss of generality we may assume that the volume $V$ of the
class $[\theta]$ is equal to one (by a trivial scaling). Consider
the following functional: 
\[
\mathcal{G}_{\beta}(u):=\mathcal{E}(u)-\mathcal{L}_{\beta}(u),\,\,\,\mathcal{L}_{\beta}(u):=\frac{1}{\beta}\log\int_{X}e^{\beta u_{\beta}}\mu_{0},
\]
 which is invariant under the additive action of $\R.$ Its critical
point equation is the ``normalized'' equation $MA(u)=e^{\beta u}\mu_{0}/\int_{X}e^{\beta u}\mu_{0},$
whose unique sup-normalized solution is given by $U_{\beta}:=u_{\beta}-\sup_{X}u_{\beta},$
where, as before, $u_{\beta}$ denotes the unique solution of the
corresponding ``non-normalized'' equation. We will use that $U_{\beta}$
is a maximizer of $\mathcal{G}_{\beta},$ as follows from a concavity
argument \cite{bbgz,berm6}. 

\emph{Step 1: Any $L^{1}-$limit point of the family $U_{\beta}$
is a maximizer of the following functional on $PSH(X,\theta):$ }
\[
\mathcal{G}_{\infty}(u):=\mathcal{E}(u)-\sup_{X}u
\]
First observe that after a harmless normalization we may as well assume
that $\mu_{0}$ is a probability measure. Then $\mathcal{L}_{\beta}(u)\leq\sup_{X}u,$
which means that $\mathcal{G}_{\beta}\geq\mathcal{G}_{\infty}.$ Hence,
for any fixed $v\in PSH(X,\theta)$ we have 
\begin{equation}
\mathcal{G}_{\beta}(U_{\beta})\geq\mathcal{G}_{\beta}(v)\geq\mathcal{G}_{\infty}(v).\label{eq:proof of conv in energy}
\end{equation}
By the compactness of $PSH(X,\theta)\Subset L^{1}(X)$ the family
$U_{\beta}$ has a limit point $U_{\infty}\in PSH(X,\theta),$ where
$U_{\infty}:=\lim_{j\rightarrow\infty}U_{\beta_{j}}$ in the $L^{1}-$topology.
Now fix $\epsilon>0.$ By the Bernstein-Markov property of $\mu_{0}$
there exists a constant $C$ such that 
\[
\mathcal{L}_{\beta}(U_{\beta})\geq\sup_{X}U_{\beta}-C/\beta-\epsilon
\]
 and hence 
\[
\mathcal{G}_{\beta}(U_{\beta})\leq\mathcal{G}_{\infty}(U_{\beta})+C/\beta+\epsilon.
\]
Finally, using that the functional $\mathcal{E}$ is usc on $PSH(X,\theta)$
and $\sup_{X}(\cdot)$ is continuous (see \cite[Cor 1.16]{b-b-w}
for a more general continuity result) it follows that 
\[
\limsup_{j\rightarrow\infty}\mathcal{G}_{\beta}(U_{\beta})\leq\mathcal{G}_{\infty}(U_{\infty})+\epsilon
\]
 which combined with the inequality \ref{eq:proof of conv in energy}
concludes the proof of the first step. 

\emph{Step two: $u_{\theta}$ is the unique sup-normalized maximizer
of $\mathcal{G}_{\infty}$ }

First note that $u_{\theta}$ maximizes \emph{$\mathcal{G}_{\infty}$
on $PSH(X,\theta).$ }To see this first observe that $u_{\theta}$
is sup-normalized, i.e. $\sup_{X}u_{\theta}=0.$ Indeed, if $\sup_{X}u_{\theta}\leq-\delta\leq0$
then $u_{\theta}+\delta\leq0$ and hence $u_{\theta}\geq u_{\theta}+\delta$
(from the very definition of $u_{\theta})$ forcing $\delta=0.$ But
if $U$ is also sup normalized, then $u_{\theta}\geq U$ and hence
$\mathcal{E}(u_{\theta})\geq\mathcal{E}(U),$ since $\mathcal{E}$
is increasing on $PSH(X,\theta),$ showing that $u_{\theta}$ is a
maximizer of $\mathcal{G}_{\infty}.$ The proof of Step two is then
concluded by using that if $u$ and $v$ are two elements in $PSH(X,\theta)$
of finite energy such that $\mathcal{E}(u)=\mathcal{E}(v),$ then
$u-v$ is a constant. This follows from the results in \cite{begz}
and can be proved as follows when$v=u_{\theta}.$ Set $\mu:=MA(u_{\theta})$
and observe that 
\[
\mathcal{E}(U)-\int U\mu\geq\mathcal{E}(U)=\mathcal{E}(u_{\theta})-\int u_{\theta}\mu,
\]
 using the orthogonality relation \ref{eq:orthog relation} in the
last equality. By concavity $u_{\theta}$ is a maximizer of the functional
$\mathcal{E}(\cdot)-\int\cdot\mu$ on $PSH(X,\theta)$ and the previous
inequality thus shows that $U$ is also a maximizer of the latter
functional. But then it follows from \cite[Thm 4.1]{bbgz} that $MA(U)=\mu$
and hence, by the uniqueness of normalized finite energy solutions
to such equations \cite[Thm A]{begz} we conclude that $U=u_{\theta},$
as desired. 

Finally, by the Bernstein-Markov property we have that $\lim_{\beta\rightarrow\infty}\mathcal{L}_{\beta}(U_{\beta})=\lim_{\beta\rightarrow\infty}\sup(U_{\beta})=0$
and hence $u_{\beta}$ also converges to $u_{\theta}$ in $L^{1}(X).$
Moreover, by Step one, we have $\mathcal{E}(u_{\beta})\rightarrow\mathcal{E}(u_{\theta}),$
which concludes the proof of the theorem. \end{proof}
\begin{rem}
The present definition of the Bernstein-Markov property is the natural
``transcendental'' generalization of the definition used in \cite[Def 1.9]{b-b-w},
which concerns the case when $[\theta]=c_{1}(L)$ for a big line bundle
$L.$ More generally, as in \cite[Def 1.9]{b-b-w} one can consider
the setting where a compact subset $K$ of $X$ has been fixed and
say that a measure $\mu_{0}$ supported on $K$ has the \emph{Bernstein-Markov
property wrt $PSH(X,\theta)$ for $K$ }if the inequality \ref{eq:BM property}
holds when $X$ has been replaced with $K.$ Repeating the proof in
the previous theorem then shows that if the latter Bernstein-Markov
property holds, then $u_{\beta}$ converges to $u_{\theta,K}$ defined
as in formula \ref{eq:def of env in text} (with $X$ replaced by
$K)$ under the condition that $u_{\theta,K}$ be continuous (i.e.
$(K,\theta)$ is regular in the sense of \cite{b-b-w}). 
\end{rem}
In the case when $[\theta]$ is a Kähler class we will only need the
$L^{1}-$convergence implicit in the previous theorem. But it should
be stressed that when we move on to the case of a big class the convergence
in energy will be crucial in order to establish the convergence in
$L^{\infty}-$norms.

\subsubsection{\label{sub:A-direct-PDE}A direct proof using the maximum principle
when $\mu_{0}$ is a volume form}

Next we show how to give an alternative direct proof of the $L^{\infty}-$convergence
towards $u_{\theta},$ which does not use the a priori regularity
result in Theorem \ref{thm:.-reg thm b-d} (on the other hand it uses
the Aubin-Yau theorem saying that $u_{\beta}$ is smooth). 
\begin{prop}
\label{prop:explicit rate in k=0000E4hler case}Let $[\theta]$ be
a Kähler class and $dV$ a volume form on $X.$ Then the correspondng
smooth solution $u_{\beta}$ of equation \ref{eq:ma eq in complex setup}
satisfies 
\[
\sup_{X}|u_{\beta}-u_{\theta}|\leq\frac{A\log\beta}{\beta},.
\]
 where the constant $A$ only depends on an upper bound on $|\theta^{n}/\omega^{n}|.$\end{prop}
\begin{proof}
Since the solution $u_{\beta}$ is smooth and $dd^{c}u_{\beta}\geq0$
at a point $x_{0}$ where the maximum of $u_{\beta}$ is attained,
equation \ref{eq:ma eq in complex setup} implies the uniform a priori
estimate 
\[
u_{\beta}\leq C/\beta,\,\,\,\,C:=\log\sup_{X}(\frac{\theta^{n}}{\omega^{n}})_{+},\,\,a_{+}:=\max\{0,a\}.
\]
Hence, $u_{\beta}-C/\beta\leq u'_{\theta}$ where $u'_{\theta}$ is
defined as $u_{\theta}$, but with the sup taken over the subspace
of all $\theta-$psh functions $u\leq0$ which are \emph{smooth}.
Conversely, fixing a smooth and strictly $\theta-$psh function $v$
and positive numbers $\epsilon$ and $\delta$ we consider a candidate
$u$ for the sup defining $u'_{\theta}$ and set $u_{\epsilon,\delta}:=(1-\epsilon)u+\epsilon v-\delta.$
Then 
\begin{equation}
(\theta+dd^{c}u_{\epsilon,\delta})^{n}\geq e^{\beta u_{\epsilon,\delta}}dV,\label{eq:subs ineq}
\end{equation}
 as long as $e^{-\delta\beta}\leq C\epsilon^{n},$ for a constant
$C$ only depending on the volume form $dV$ (and the fixed element
$v).$ In particular, the previous inequality holds for $\epsilon=1/\beta$
and $\delta=\frac{C'}{\beta}\log\beta$ for $C'$ sufficently large.
But then, comparing the inequality \ref{eq:subs ineq} and the defining
equation \ref{eq:ma eq in complex setup}, it follows from the maximum
principle that $u_{\epsilon,\delta}\leq u_{\beta}$ (see Lemma\ref{lem:cons of compar principle in khler case}).
All in all this means that 
\[
u_{\beta}-C/\beta\leq u'_{\theta}\leq\frac{1}{(1-1/\beta)}u_{\beta}+\frac{C'}{\beta}\log\beta,
\]
 and hence the proof is concluded by the observation that $u'_{\theta}=u{}_{\theta},$
which is an immediate consequence of Demailly's regularization theorem.
In fact, it is not necessary to invoke the latter regularization result
as the argumentent above leads to a new PDE proof of it, as explained
in Section \ref{sub:Applications-to-regularization}.
\end{proof}

\subsubsection{$L^{\infty}-$ estimates}

We start with the following well-known
\begin{lem}
\label{lem:cons of compar principle in khler case}Assume that $u$
and $v$ are (say, bounded) $\theta-$ psh functions such that $MA(v)\geq e^{\beta v}dV$
and $MA(u)\leq e^{\beta u}dV.$ Then $v\leq u.$\end{lem}
\begin{proof}
In the smooth case this follows immediately from the maximum principle
and in the general case we can apply the comparison principle (which
will be important in the setting of big class considered below). Indeed,
according to the comparison principle $\int_{\{u\leq v\}}MA(v)\leq\int_{\{u\leq v\}}MA(u)$
and hence $\int_{\{u\leq v\}}e^{\beta v}dV\leq\int_{\{u\leq v\}}e^{\beta u}dV.$
But then it must be that $v\leq u$ a.e. on $X$ and hence everywhere.
\end{proof}
The previous lemma allows us to construct ``barriers'' to show that
$u_{\beta}$ is uniformly bounded:
\begin{lem}
\label{lem:unif estimate in the khler case}There exists a constant
$C$ such that $\sup_{X}|u_{\beta}|\leq C.$\end{lem}
\begin{proof}
Let us start with the proof of the lower bound on $u_{\beta}.$ Since
$[\theta]$ is a Kähler class there is a smooth $\theta-$psh function
$v$ such that $MA(v)\geq e^{-C}dV$ for some constant $C.$ After
shifting $v$ by a constant we may assume that $v\leq-C/\beta.$ But
then $MA(v)\geq e^{-C}dV\geq e^{\beta v}$ and hence by the previous
lemma $v\leq u_{\beta}$ which concludes the proof of the lower bound.
Similarly, taking $v$ to be a smooth $\theta-$psh function $v$
such that $MA(v)\leq e^{C}dV$ and shifting $v$ so that $C/\beta\leq v$
proves that $u_{\beta}\leq v,$ which concludes the proof of the lemma. 
\end{proof}

\subsubsection{The Laplacian estimate}

Next we will establish the following key Laplacian estimate:
\begin{prop}
\label{prop:laplace estimate in khler case}Fix a Kähler form $\omega$
in $[\theta].$ Then there exists a constant $C$ such that, for $\beta\geq\beta_{0},$
\[
-C\leq\Delta_{\omega}u_{\beta}\leq C
\]
\end{prop}
\begin{proof}
The lower bound follows immediately from $\theta+dd^{c}u_{\beta}\geq0.$
To prove the upper bound we first recall the following variant of
the Aubin-Yau Laplacian estimate in this context due to Siu (compare
page 99 in \cite{siu} and Prop 2.1 in \cite{cgh}): given two Kähler
forms $\omega'$ and $\omega$ such that $\omega'^{n}=e^{f}\omega^{n}$
we have that 
\[
\Delta_{\omega'}\log tr_{\omega}\omega'\geq\frac{\Delta_{\omega}f}{tr_{\omega}\omega'}-Btr_{\omega'}\omega,
\]
 where the constant $B$ is proportional to the infimum of the holomorphic
bisectional curvatures of $\omega.$ Fixing $\beta>0$ and setting
$\omega':=\theta+dd^{c}u$ for $u:=u_{\beta}$ we have, by the MA-equation
for $u_{\beta},$ that $f=\beta u$ and hence 
\[
Btr_{\omega'}\omega+\Delta_{\omega'}\log tr_{\omega}\omega'\geq\beta\frac{\Delta_{\omega}u}{tr_{\omega}\omega'}
\]
Next, we note that $\Delta_{\omega}u=tr_{\omega}\omega'-tr_{\omega}\theta.$
Moreover, writing $\omega=\omega'-dd^{c}(u-v),$ where $v$ is a smooth
function such that 
\begin{equation}
\omega=\theta+dd^{c}v,\label{eq:def of v as kahler pot in proof}
\end{equation}
 also gives $tr_{\omega'}\omega=n-\Delta_{\omega'}(u-v).$ Accordingly,
the previous inequality may be reformulated as follows:

\[
nB+\Delta_{\omega'}(\log tr_{\omega}\omega'-B(u-v))\geq\beta\frac{tr_{\omega}\omega'-tr_{\omega}\theta}{tr_{\omega}\omega'},
\]
 and hence (letting $C$ be the sup of $tr_{\omega}\theta)$ 
\begin{equation}
(C\beta+nBtr_{\omega}\omega')e^{-B(u-v)}+\Delta_{\omega'}\log(tr_{\omega}\omega'-B(u-v))tr_{\omega}\omega'e^{-B(u-v)}\geq\beta tr_{\omega}\omega'e^{-B(u-v)}\label{eq:key laplacian bound}
\end{equation}
Thus, setting $s:=\sup_{X}e^{-B(u-v)}tr_{\omega}\omega'$ and taking
the maximum over $X$ in the previous inequality gives 
\[
\beta s\leq0+nBs+\beta\sup_{X}Ce^{-B(u-v)}
\]
Finally, by the previous lemma $u(:=u_{\beta})$ is uniformly bounded
in $x$ and $\beta$ and since, by definition $v$ is bounded, it
follows that $tr_{\omega}\omega'$ is uniformly bounded from above,
as desired. More precisely, the previous argument gives the estimate
\begin{equation}
tr_{\omega}\omega'\leq\frac{1}{1-nB/\beta}e^{B(u-v)}\left(nB/\beta+\sup_{X}\left(tr_{\omega}\theta\right)e^{-\inf_{X}B(u-v)}\right)\label{eq:precise laplcestimate}
\end{equation}
\end{proof}
\begin{rem}
Note that, in general, the Ricci curvature of the Kähler forms $\omega_{\beta}:=\theta+dd^{c}u_{\beta}$
is unbounded, both from above and below, as $\beta\rightarrow\infty.$
Still, by the previous estimate, the Kähler forms $\omega_{\beta}$
are uniformly bounded from above. However it should be stressed that,
unless $\theta>0,$ there is no uniform bound of the form $\omega_{\beta}\geq\delta\omega>0$
as it will follow from Theorem \ref{thm:main thm khler case intro}
that $\omega_{\beta}^{n}\rightarrow0$ on large portions of $X$ (indeed,
for $\beta$ large, $\omega_{\beta}^{n}\leq Ce^{-\beta\epsilon}dV$
on the open set where $u_{\theta}<-2\epsilon).$
\end{rem}

\subsubsection{End of proof of Theorem \ref{thm:main thm khler case intro} using
the variational approach}

By Lemma \ref{lem:unif estimate in the khler case} $u_{\beta}$ is
uniformly bounded and by the Laplacian estimate in Prop \ref{prop:laplace estimate in khler case}
combined with Green's formula the gradients of $u_{\beta}$ are uniformly
bounded. Hence, it follows from basic compactness results that, after
perhaps passing to a subsequence, $u_{\beta}$ converges to a function
$u$ in $\mathcal{C}^{1,\alpha}(X)$ for any fixed $\alpha\in]0,1[.$
It will thus be enough to show that $u=u_{\theta}$ (since this will
show that any limit point of $\{u_{\beta}\}$ is uniquely determined
and coincides with $u_{\theta}$). But this follows from either Theorem
\ref{thm:conv in energy khler case} or Proposition \ref{prop:explicit rate in k=0000E4hler case}.

\subsection{\label{sub:The-case-of big class}The case of a big class (proof
of Theorem \ref{thm:main thm big case intro})}

A (Bott-Chern) cohomology class $[\theta]$ in $H^{1,1}(X)$ is said
to be\emph{ big, }if $[\theta]$ contains a Kähler current $\omega,$
i.e. a positive current $\omega$ such that that $\omega\geq\epsilon\omega_{0}$
for some positive number $\epsilon,$ where $\omega_{0}$ is a fixed
strictly positive form $\omega_{0}$ on $X.$ We also recall that
a class $[\theta]$ is said to be \emph{nef }if, for any $\epsilon>0,$
there exists a smooth form $\omega_{\epsilon}\in T$ such that $\omega_{\epsilon}\geq-\epsilon\omega_{0}.$
To simplify the exposition we will assume that $X$ is a Kähler manifold
so that the form $\omega_{0}$ may be chosen to closed. Then the cone
of all big classes in the cohomology group $H^{1,1}(X)$ may be defined
as the interior of the cone of pseudo-effective classes and the cone
of Kähler classes may be defined as the interior of the cone of nef
classes.

We also recall that a function $u$ in $PSH(X,\theta)$ is said to
have\emph{ minimal singularities,} if for any $v\in PSH(X,\theta)$
the function $u-v$ is bounded from below on $X.$ In particular,
the envelope $u_{\theta}$ has (by its very definition) minimal singularities
(and this is in fact the standard construction of a function with
minimal singularities). In the case when $[\theta]$ is big any function
with minimal singularities is locally bounded on a Zariski open subset
$\Omega,$ as a well-known consequence of Demailly's approximation
results \cite{dem0}. In fact, the subset $\Omega$ can be taken as
the \emph{Kähler (ample) locus} of $[\theta]$ defined in \cite{bo}.
\begin{example}
\label{ex: nef and big}Let $Y$ be a singular algebraic variety in
complex projective space $\P^{N}$ and $\omega$ a Kähler form on
$\P^{n}$ (for example, $\omega$ could be taken as the Fubini-Study
metric so that $[\omega_{|Y}]$ is the first Chern class of $\mathcal{O}_{X}(1)$).
If now $X\rightarrow Y$ is a smooth resolution of $Y,$ which can
be taken to invertible over the regular locus of $Y;$ then the pull-back
of $\omega$ to $X$ defines a class which is nef and big and such
that its Kähler locus corresponds to the regular part of $Y.$ 
\end{example}
We will denote by $MA$ the Monge-Ampère operator on $PSH(X,\theta)$
defined by replacing wedge products of smooth forms with the \emph{ non-pluripolar
produc}t of positive currents introduced in \cite{begz}. The corresponding
operator $MA$ is usually referred to as the \emph{non-pluripolar
Monge-Ampère operator. }For example, if $u$ has minimal singularities,
then $MA(u)=1_{\Omega}MA(u_{|\Omega})$ on the Kähler locus $\Omega,$
where\emph{ $MA(u_{|\Omega})$ }may be computed locally using the
classical definition of Bedford-Taylor. We let $V$ stand for the
volume of the class $[\theta],$ which may be defined as the total
mass of $MA(u)$ for any function $u$ in $PSH(X,\theta)$ with minimal
singularities. By \cite{begz} there exists a unique solution $u_{\beta}$
to the equations \ref{eq:ma eq in complex setup} in $PSH(X,\theta)$
with minimal singularities. Moreover, by \cite{begz} the solution
is smooth on the Kähler locus in the case when $[\theta]$ is nef
and big (which is expected to be true also without the nef assumption;
compare the discussion in \cite{begz}).

\subsubsection{Convergence in energy}

In the case of a big class one first defines, following \cite{bbgz},
the following functional on the space of all functions in n $PSH(X,\theta)$
with minimal singularities: 

\begin{equation}
\mathcal{E}(u):=\frac{1}{n+1}\int_{X}\sum_{j=0}^{n}(u-u_{\theta})(\theta+dd^{c}u)^{j}\wedge(\theta+dd^{c}u_{\theta})^{n-j}\label{eq:energy functional in big case}
\end{equation}
(the point is that we needs to subtract $u_{\theta}$ to make sure
that the integral is finite). Equivalently, $\mathcal{E}$ may be
defined as the primitive of the Monge-Ampère operator on the the space
of all finite energy functions in $PSH(X,\theta),$ normalized so
that $\mathcal{E}(u_{\theta})=0.$ We then define convergence in energy
as before.
\begin{rem}
Strictly speaking, in the case of a Kähler class the definition \ref{eq:energy functional in big case}
of $\mathcal{E}$ only coincides with the previous one (formula \ref{eq:def of energy functional in khler case})
in the case when $\theta$ is semi-positive (since the definition
in formula \ref{eq:def of energy functional in khler case} corresponds
to the normalization condition $\mathcal{E}(0)=0).$ But the point
is that, in the Kähler case, different normalizations gives rise to
functionals which only differ up to an overall additive constant and
hence the choice of normalization does not effect the notion of convergence
in energy.
\end{rem}
The proof of Theorem \ref{thm:main thm khler case intro} can now
be repeated word for word to give the following 
\begin{prop}
\label{prop:conv in energy big case}Suppose that $\theta$ is a smooth
form such that the class $[\theta]$ is big. Then $u_{\beta}$ converges
to $u_{\theta}$ in energy.
\end{prop}

\subsubsection{$L^{\infty}-$estimates.}

We will also need the following upper bound on $u_{\beta}:$
\begin{lem}
\label{lem:refined upper bound}There exists a constant $C$ such
that
\[
u_{\beta}\leq u_{\theta}+C/\beta
\]
(the constant $C$ may be taken as $\log(\theta^{n}/dV)_{+},$ where
$a_{+}:=\max\{0,a\}).$ \end{lem}
\begin{proof}
We recall that if $u_{\beta}$ is smooth (as in the case of a Kähler
class) then the inequality follows directly from the maximum principle.
In the general case the inequality follows from the fact that $u_{\beta}$
is a viscosity subsolution of the equation \ref{eq:ma eq in complex setup},
as follows from the results in \cite{e-g-z}. Indeed, first assume
that the maximum of $u_{\beta}$ on $X$ is achieved at a point $x_{0}$
in the Zariski open subset $\Omega$ (defined as the Kähler locus
of the class $[\theta]$). Then we can introduce local holomorphic
coordinates centered at $x_{0}$ and locally write $\theta=dd^{c}f$
for $f$ smooth and set $\phi:=u_{\beta}+f,$ which defines a locally
bounded psh function $\phi.$ The defining equation for $u_{\beta}$
implies the following local inequality, say on a neighourhood of the
the ball $B\subset\C^{n}:$ 
\[
(dd^{c}\phi)^{n}\geq e^{\beta(\phi-f)}dV
\]
in the pluripotential sense of Bedford-Taylor (in fact, equality holds,
but we will only need the inequality above). Moreover, by assumption
$\phi-f$ has a local maximum at $0.$ But then it follows from local
considerations (based on the Bedford-Taylor comparison principle for
bounded psh functions) that 
\[
e^{\beta(\phi-f)}dV\leq(dd^{c}f)^{n}\,\,\,\mbox{at\,}z=0,
\]
(see \cite[Prop 1.11]{e-g-z}). In other words, 
\[
u_{\beta}\leq C_{0}/\beta,\,\,\,C_{0}=\log(\theta^{n}/dV)_{+}.
\]
 which proves the lemma in this case. In the general case we fix a
sup-normalized function $v\in PSH(X,\theta)$ wich is smooth on $\Omega$
and such that $v-u_{\theta}\rightarrow-\infty$ along the analytic
subvariety $X-\Omega.$ Given $\epsilon>0$ we set $u_{\beta,\epsilon}:=(1-\epsilon)u_{\beta}+\epsilon v\in PSH(X,\theta)$
which is locally bounded on $\Omega$ and satisfies the following
inequality in the sense of Bedford-Taylor on $\Omega$ 
\[
MA_{\theta}(u_{\beta,\epsilon})\geq(1-\epsilon)^{n}e^{\beta u_{\beta}}dV\geq(1-\epsilon)^{n}e^{\beta_{\epsilon}u_{\beta,\epsilon}}dV,\,\,\,\beta_{\epsilon}:=\beta(1-\epsilon)^{-1}
\]
 using that $v\leq0$ in the last inequality. By assumption there
exists a point $x_{\epsilon}$ in $\Omega$ where $u_{\beta,\epsilon}$
achieves its maximum. Hence, we can apply the previous argument to
$\phi:=u_{\beta,\epsilon}+f$ with parameter $\beta_{\epsilon}$ to
get an inequality of the form $u_{\beta,\epsilon}\leq C_{\epsilon}/\beta_{\epsilon},$
where $C_{\epsilon}\rightarrow C_{0}$ as $\epsilon\rightarrow0.$
Letting $\epsilon$ tend to zero thus concludes the proof of the lemma.
\end{proof}
We recall that in the case of a Kähler class the estimate in the previous
lemma was obtained as consequence of the maximum principle in the
proof of Proposition \ref{prop:explicit rate in k=0000E4hler case}.
Next, we generalize the $L^{\infty}-$convergence in Proposition \ref{prop:explicit rate in k=0000E4hler case}
to a general big class, using the convergence in energy in Prop \ref{prop:conv in energy big case}.
\begin{prop}
\label{prop:uniform convergence in big case}Suppose that $\theta$
is a smooth form such that the class $[\theta]$ is big. Then $u_{\beta}$
converges uniformly to $u_{\theta}$ on $X,$ i.e. 
\[
\lim_{\beta\rightarrow0}\left\Vert u_{\beta}-u_{\theta}\right\Vert _{L^{\infty}(X)}=0
\]
\end{prop}
\begin{proof}
According to the previous lemma we have that $u_{\beta}\leq u_{\theta}+C/\beta$
and hence $MA(u_{\beta})/dV\leq e^{C}.$ Moreover, by Prop \ref{prop:conv in energy big case}
$u_{\beta}$ converges to $u_{\theta}$ in energy. As will be next
explained these properties are enough to conclude that $u_{\beta}$
converges uniformly to $u.$ Indeed, it is well-known that if $u_{j}$
is a sequence in $PSH(X,\theta)$ converging in capacity to $u_{\infty}$
with a uniform bound $L^{p}-$bound on $MA(u_{j})/dV,$ then $\left\Vert u_{j}-u_{\infty}\right\Vert _{L^{\infty}(X)}\rightarrow\infty,$
as follows from a generalization of Kolodziej's $L^{\infty}-$estimates
to the setting of a big class (see \cite{begz,g-z2} and references
therein). Finally, as shown in \cite{bbgz}, convergence in energy
implies convergence in capacity, which thus concludes the proof of
the previous proposition. In fact, using the stability results in
\cite{g-z2} a more quantitative convergence result can be given.
Indeed, according to Prop 4.2 in \cite{g-z2} the following holds:
assume that $\varphi$ and $\psi$ are functions in $PSH(X,\theta)$
normalized so that $\sup\varphi=\sup\psi=0$ and such that $MA(\varphi)\leq fdV,$
where $f\in L^{p}(X,dV).$ Then, for any sufficiently small positive
number $\gamma$ (see \cite{g-z2} for the precise condition) there
exists a constant $M,$ only depending on $\gamma$ and an upper bound
on $\left\Vert f\right\Vert _{L^{p}(dV)},$ such that 
\[
\sup_{X}(\psi-\varphi)^{+}\leq M\left\Vert (\psi-\varphi)^{+}\right\Vert _{L^{1}(X,MA(\varphi))}^{\gamma}
\]
Setting $\varphi:=u_{\beta}-\epsilon_{\beta},$ where $\epsilon_{\beta}=\sup u_{\beta}$
and $\psi:=u_{\theta}$ thus gives, for $\gamma,$ fixed 
\[
\sup_{X}(u_{\theta}-u_{\beta}-\epsilon_{\beta})^{+}\leq M\left(\int\left|u_{\theta}-u_{\beta}-\epsilon_{\beta}\right|MA(u_{\beta})\right)^{\gamma}
\]
Now, by the convergence in energy and the $L^{1}-$convergence in
Prop \ref{prop:conv in energy big case} we have 
\[
\int(u_{\beta}-u_{\theta})MA(u_{\beta})\rightarrow0
\]
and since $\left|\int u_{\theta}-u_{\beta}-\epsilon_{\beta}\right|MA(u_{\beta})\leq\int(u_{\theta}-u_{\beta}-C/\beta)MA(u_{\beta})+C/\beta+\epsilon_{\beta}$
we deduce that $\sup_{X}(u_{\theta}-u_{\beta}-\epsilon_{\beta})^{+}\rightarrow0,$
i.e. $u_{\theta}\leq u_{\beta}+\epsilon'_{\beta},$ which concludes
the proof.
\end{proof}

\subsubsection{\label{sub:Bound-on-the}Bound on the Monge-Ampère measure of $u_{\theta}$}

As shown above $u_{\beta}$ converges to $u_{\theta}$ in energy (and
even uniformly). In particular, the convergence holds weakly for the
corresponding Monge-Ampère measures. The bound in Lemma \ref{lem:refined upper bound}
thus implies that
\[
MA(u_{\theta})\leq\sup_{X}\left(\frac{(\theta^{n})_{+}}{dV}\right)dV
\]
for any given volume form $dV$ on $X.$ Taking a sequence of volume
forms $dV_{\epsilon}$ approximating the measure $(\theta^{n})_{+}$
thus gives $MA(u_{\theta})\leq(\theta^{n})_{+}$ on $X.$ Since $MA(u_{\theta})$
is supported on the coincidence set $D$ (which is contained in the
set where $\theta\geq0)$ this proves the inequality \ref{eq:bound on ma intro}.

\subsubsection{Laplacian estimates}

For the Laplacian estimate we will have to assume that the big class
$[\theta]$ is nef.
\begin{prop}
\label{prop:laplace in nef and big}Suppose that the class $[\theta]$
is nef and big. Then the Laplacian of $u_{\beta}$ is locally bounded
wrt $\beta$ on the Zariski open set $\Omega\subset X$ defined as
the Kähler locus of $X.$ \end{prop}
\begin{proof}
We will assume that $X$ is a Kähler manifold, i.e. $X$ admits some
Kähler form $\omega_{0}$ (not necessarily cohomologous to $\theta).$
Then $\theta$ is nef precisely when the class $[\theta]+\epsilon[\omega_{0}]$
is Kähler for any $\epsilon>0.$ Setting $\theta_{\epsilon}:=\theta+\epsilon\omega_{0}$
and fixing $\epsilon>0$ and $\beta>0$ we denote by $u_{\beta,\epsilon}$
the solutions of the Monge-Ampère equations obtained by replacing
$\theta$ with $\theta_{\epsilon}.$ Then it follows from well-known
results \cite{begz} that, as $\epsilon\rightarrow0,$ 
\[
u_{\beta,\epsilon}\rightarrow u_{\beta}\,\,\,\mbox{in\,}\mathcal{C}_{loc}^{\infty}(\Omega).
\]
 Moreover, since $[\theta]$ is assumed big there exists a positive
current $\omega$ in $[\theta]$ such that the restriction of $\omega$
to $\Omega$ coincides with the restriction of a Kähler form on $X.$
More precisely, we can take $\omega$ to be a Kähler current on $X$
such that $\omega=dd^{c}v+\theta$ for a function $v$ on $X$ such
that $v$ is smooth on $\Omega$ and $u-v\rightarrow-\infty$ at the
``boundary'' of $\Omega$ (using that $u$ has minimal singularities;
compare \cite{begz} ). Setting $u:=u_{\beta,\epsilon}$ the inequality
\ref{eq:key laplacian bound} still applies on $\Omega.$ Moreover,
since $u-v\rightarrow-\infty$ at the boundary of $\Omega$ the sup
$s$ defined above is attained at some point of $\Omega$ and $\sup_{X}Ce^{-B(u-v)}\leq C'.$
Accordingly, we deduce that 
\[
s:=\sup_{X}e^{-B(u-v)}tr_{\omega}\omega'\leq C''
\]
precisely as before, which in particular implies that $tr_{\omega}(\theta+dd^{c}u_{\beta,\epsilon})$
is locally bounded from above (wrt $\beta$ and $\epsilon).$ Finally,
letting $\epsilon\rightarrow0$ concludes the proof.
\end{proof}
In the special case when $\theta$ is semi-positive and big (the latter
condition then simply means that $V>0)$ it follows from the results
in \cite{e-g-z} that $u_{\beta}$ is continuous on all of $X$ and
hence Prop \ref{prop:uniform convergence in big case} then says that
$u_{\beta}\rightarrow u_{\theta}$ in $C^{0}(X).$
\begin{rem}
\label{rem: apriori}The precise Laplacian estimate obtained in the
previous proof may, for $v$ and $\omega$ as in the proof above may
be formulated as 
\begin{equation}
tr_{\omega}\omega_{u_{\beta}}\leq\frac{1}{1-nB/\beta}e^{B(u_{\beta}-v)}\left(B/\beta+\sup_{X}\left(tr_{\omega}\theta\right)e^{-\inf_{X}B(u_{\beta}-v)}\right)\label{eq:precise laplcestimate-1}
\end{equation}
In particular, normalizing $v$ so that $\sup_{X}v=0$ gives 
\[
tr_{\omega}\omega_{u_{\beta}}\leq\frac{e^{\sup u_{\beta}-\inf u_{\beta}}}{1-nB/\beta}e^{-Bv}\left(B/\beta+\sup_{X}\left(tr_{\omega}\theta\right)\right)
\]
By the $L^{\infty}-$estimates above $\sup_{X}u_{\beta}-\inf_{X}u_{\beta}$
is uniformly bounded in terms of $\sup_{X}\left|\theta^{n}/dV\right|.$
In particular, letting $\beta\rightarrow\infty$ gives the following
a priori estimate for the Laplacian of the envelope $u_{\theta}:$
\begin{equation}
tr_{\omega}\omega_{u_{\theta}}\leq Ce^{-Bv},\label{eq:a priori estimate on curv of env}
\end{equation}
 where the constant $C$ only depends on an upper bound on $|\theta|_{\omega}.$
Interestingly, the estimate \ref{eq:a priori estimate on curv of env}
is essentially of the same form as the one obtained in \cite{b-d},
in the more general setting of a big class, by a completely different
method where the constant $B$ (i.e. the lower bound on the bisectional
curvature) arises in the initial step of the proof where the envelope
is regularized by the global convolution type operator associated
to the exponential flow determined by the Chern connection. 
\end{rem}

\subsubsection{End of the proof of Theorem \ref{thm:main thm big case intro} in
the big case}

This is proved exactly as in the case of a Kähler class, given the
convergence results established above.

\section{\label{sub:Applications-to-regularization}Transcendental Bergman
metric asymptotics and Applications to regularization of $\omega-$psh
functions }

\subsection{\label{sub:Transcendental-Bergman-kernels}Transcendental Bergman
kernels }

Consider an ample line bundle $L\rightarrow X$ and a pair $(\left\Vert \cdot\right\Vert ,dV)$
consisting of an Hermitian metric $\left\Vert \cdot\right\Vert $
on $L$ and a volume form $dV$ on $X.$ We denote by $\theta$ the
normalized curvature form of $\left\Vert \cdot\right\Vert ,$ which
represents the first Chern class $c_{1}(L)$ in $H^{1,1}(X,\R)\cap H^{2}(X,\Z).$
The corresponding \emph{Bergman function} $\rho_{k}$ (also called
the\emph{ density of states function}), at level $k,$ may be defined
\[
\rho_{k}(x)=\sum_{i=1}^{N_{k}}\left\Vert s_{i}^{(k)}(x)\right\Vert ^{2},
\]
 in terms of any fixed basis $s_{i}^{(k)}$ in $H^{0}(X,L^{\otimes k})$
which is orthonormal wrt the corresponding $L^{2}-$norm determined
by the pair $(\left\Vert \cdot\right\Vert ,dV).$ In other words,
$\rho_{k}(x)$ is the restriction to the diagonal of the squared point-wise
norm of the Bergman kernel of $H^{0}(X,L^{\otimes k})$ (see \cite{berm1}
and references therein). The function $v_{k}:=\frac{1}{k}\log\rho_{k}$
is often referred to as the\emph{ Bergman metric (potential)} at level
$k,$ determined by $(\left\Vert \cdot\right\Vert ,dV)$ (geometrically,
$\left\Vert \cdot\right\Vert e^{-kv_{k}}$ is the pull-back of the
Fubini-Study metric on the projective space $\P H^{0}(X,L^{\otimes k})$
under the corresponding Kodaira embedding). As shown in \cite{berm1}
the corresponding \emph{Bergman measures
\[
\nu_{k}:=\frac{1}{N_{k}}\rho_{k}(x)dV
\]
} converge weakly to $MA_{\theta}(u_{\theta})$ and $v_{k}$ converges
uniformly to $u_{\theta}.$ In particular, 
\[
MA_{\theta}(v_{k})\thickapprox e^{kv_{k}}dV
\]
 in the sense that both measures have the same weak limit (namely
$MA_{\theta}(u_{\theta})).$ We can thus view the Bergman metric $v_{k}$
as an approximate solution to the equation \ref{eq:ma eq in intro},
for $\beta=k.$ This motivates thinking of the family $u_{\beta}$
of\emph{ exact }solutions, defined with respect to a general smooth
closed $(1,1)-$form $\theta$ (not necessarily corresponding to a
line bundle) as a\emph{ transcendental Bergman metric, }in the sense
that it behaves (at least asymptotically as $\beta\rightarrow\infty$)
as a Bergman metric associated to an Hermitian line bundle. Similarly,
$e^{ku_{\beta}}dV(=MA_{\theta}(u_{\beta}))$ can be thought of as
a transcendental Bergman measure. 

The main virtue of the family $u_{\beta}$ is that it is canonically
determined by the pair $(\theta,dV))$ and exists also in the general
transcendental setting of a Kähler class $[\theta]$ which can not
be realized as the first Chern class $c_{1}(L)$ of a line bundle.
Accordingly, it seems natural to expect that it can be used as a substitute
for the timehonoured technique in complex geometry of using Bergman
kernels as an approximation tool. In Section \ref{sub:Regularization-of-psh}
and Section \ref{sec:Applications-to-geodesic} we will give two such
applications to the regularization problem of $\omega-$psh functions
and weak geodesic rays, respectively. 

In the following it will be convenient to use the equivalent formulation
of envelopes of the form $P_{\omega}(f)$ in Section \ref{sub:An-alternative-formulation}
(occasionally dropping the subscript $\omega).$ In other words, we
start with a reference Kähler form $\omega$ on $X.$ Given a smooth
function $f$ we denote by $P_{\beta}(f)$ the solution $\varphi_{\beta}$
of the corresponding Monge-Ampère equation \ref{eq:ma equation with k=0000E4hler ref}.
In the line bundle setting above this corresponds to fixing a reference
metric $\left\Vert \cdot\right\Vert _{0}$ on $L$ and writing $\left\Vert \cdot\right\Vert ^{2}=\left\Vert \cdot\right\Vert _{0}e^{-f}$
wich has curvature form $\theta=\omega+dd^{c}f.$ 
\begin{lem}
\label{lem:The-operator-P beta}The operator $P_{\beta}:\,C^{\infty}(X)\rightarrow SPSH(X,\omega)\cap C^{\infty}(X$
is decreasing, i.e. if $f\leq g,$ then $P_{\beta}f\leq P_{\beta}g.$
Moreover, $P_{\beta}(f+c)=P_{\beta}(f)+c$ for any $c\in\R$ and hence
\begin{equation}
\left\Vert P_{\beta}f-P_{\beta}g\right\Vert _{L^{\infty}(X)}\leq\left\Vert f-g\right\Vert _{L^{\infty}(X)}.\label{eq:contr prop of P beta}
\end{equation}
 \end{lem}
\begin{proof}
The decreasing property follows directly from the comparison principle
(Lemma \ref{lem:cons of compar principle in khler case}) and the
scaling property from the very definitions of $P_{\beta}.$ 
\end{proof}
By Prop \ref{prop:explicit rate in k=0000E4hler case} $P_{\beta}$
converges to the projection operator $P:$ 
\begin{equation}
\left\Vert P_{\beta}f-Pf\right\Vert _{L^{\infty}(X)}\leq\frac{A\log\beta}{\beta},.\label{eq:conv of Pbeta to P}
\end{equation}
where the constant $A$ only depends on an upper bound on $(\omega+dd^{c}f)^{n}.$
In particular, by a simple approximation argument (using \ref{eq:contr prop of P beta})
$P_{\beta}f$ converges to $f$ uniformly, for any continuous function
$f$ on $X.$ These convergence results can be viewed as transcental
analogs of the Bergman metric asymptotics in \cite{berm1} (which
has the corresponding rate with $\beta=k).$ Moreover, for $f$ continuous
the corresponding weak convergence of the transcendental Bergman measures:
\[
\lim_{\beta\rightarrow\infty}e^{k(P_{\beta}f-f)}dV=(\omega+dd^{c}Pf)^{n}
\]
(resulting from the convergence of Monge-Ampère measues) is the analog
of the convergence of Bergman measures towards equilibrium measures
in \cite{berm1} (first shown by Bouche and Tian, independently, in
the case of a smooth and metrics with strictly positive curvature
form $\theta).$
\begin{rem}
Let us briefly explain how the setting above fits into the statistical
mechanical setup recalled in Section \ref{sub:Further-background-and}.
The point is that one can let the inverse temperature $\beta,$ defining
the probability measures \ref{eq:prob measure}, depend on $k.$ In
particular, for $\beta=k$ one obtains a \emph{determinantal} random
point process. A direct calculation (compare \cite{berm3}) reveals
that the corresponding one point correlation measure $\int_{X^{N_{k-1}}}\mu_{k,\beta}$
then coincides with the Bergman measure $\nu_{k}$ defined above.
This means that the limit $k\rightarrow\infty$ which appears in the
``Bergman setting'' can - from a statistical mechanical point of
view - be seen as a limit where the number $N_{k}$ of particles\emph{
and }the inverse temperature $\beta$ jointly tend to infinity. 
\end{rem}

\subsection{\label{sub:Regularization-of-psh}Regularization of $\omega-$psh
functions}

In this section we consider the case of a Kähler class $[\omega].$
We show how to give a simple global PDE proof of the following special
case of the general regularization results of Demailly \cite{dem0}:
\begin{thm}
\label{thm:Demailly}Let $[\omega]$ be a Kähler class. Then any function
$\psi\in PSH(X,\omega)$ can be written as a decreasing limit of functions
$\psi_{j}$ which are smooth and strictly $\omega-$psh. \end{thm}
\begin{proof}
Since $\psi$ is usc we can write it as a decreasing limit of smooth
functions $f_{j}.$ Setting 
\begin{equation}
(P'_{\omega}f)(x):=\sup\left\{ \varphi(x):\,\,\varphi\leq f:\,\,\varphi\in PSH(X,\omega)\cap C^{\infty}\right\} \label{eq:def of proj wrt smooth}
\end{equation}
we note that the sequence $\varphi_{j}:=P'_{\omega}f_{j}$ decreases
to $\psi.$ Indeed, since the operator $P'_{\omega}$ is decreasing
the sequence $\varphi_{j}$ is decreasing and $\varphi_{j}\geq\psi.$
Moreover, fixing a point $x$ and $\epsilon>0$ we have that $\varphi_{j}(x)\leq f_{j}(x)\leq\psi(x)+\epsilon$
for $j\geq j_{\epsilon}$ showing that $\varphi_{j}(x)$ decreases
to $\psi(x)$ for any $x,$ as claimed. Next, fixing $\beta>0$ we
set $\varphi_{j,\beta}:=P_{\beta}f_{j}$ converging uniformly to $\varphi_{j,\beta}$
as $\beta\rightarrow\infty$ (by Prop \ref{prop:explicit rate in k=0000E4hler case};
compare formula \ref{eq:conv of Pbeta to P}). Hence, for appropriate
choices of sequence $\epsilon_{j}\rightarrow0$ and $\beta_{j}\rightarrow\infty$
the sequence $\psi_{j}:=\varphi_{j,\beta_{j}}+\epsilon_{j}$ has the
desired property (and as a consequence we actually have $P'_{\omega}f=P{}_{\omega}f,$
by approximation).
\end{proof}
It should be pointed out that by a local gluing argument of Richberg
\cite{ri} the regularization result above can be reduced to the case
of a \emph{continuous} $\omega-$psh function $\psi$ (using the usual
local regularizations involving convolutions). In turn, it was shown
in \cite{bl-k} that the continuity assumption can be replaced by
the assumption of vanishing Lelong numbers and hence, as explained
in \cite{bl-k}, approximating a general element $\psi\in PSH(X,\omega)$
with the decreasing sequence $\psi_{l}:=\max\{\psi,l\}$ in $PSH(X,\omega)\cap L^{\infty}$
gives a simple elemenary proof of the previous theorem. In the light
of the discussion in the previous section the present global regularization
scheme can be seen as a transcental analog of the well-known Bergman
kernel approach to regularization used in the line bundle setting
(see \cite{dem0,g-z}). The present approach has the virtue of preserving
higher order regularity properties of $\psi$ as summarized in the
following
\begin{thm}
Let $(X,\omega)$ be a compact Kähler manifold $\varphi$ an $\omega-$psh
function such that its Monge-Ampère measure $(\omega+dd^{c}\varphi)^{n}$
has an $L^{\infty}-$density. Then $\varphi_{\beta}:=P_{\beta}(\varphi)$
is in $PSH(X,\omega)\cap C^{2,\alpha}$ for some $\alpha>0$ and satisfies
\[
\sup_{X}\left|\varphi_{\beta}-\varphi\right|\leq C\frac{\log\beta}{\beta},\,\,\,\,\,(\omega+dd^{c}\varphi_{\beta})^{n}\leq C\omega^{n}
\]
where the constant $C$ only depends on an upper bound on the density
$(\omega+dd^{c}\varphi)^{n}/\omega^{n}.$ Moreover, if the positive
current $(\omega+dd^{c}\varphi)$ has coefficents in $L^{\infty}$
then $\omega+dd^{c}\varphi_{\beta}\leq C'\omega$ and $\varphi_{\beta}$
is in $PSH(X,\omega)\cap C^{3,\alpha}$ for any $\alpha<1.$ \end{thm}
\begin{proof}
Since $(\omega+dd^{c}f)^{n}$ has an $L^{\infty}-$density \cite{ko}
gives that $\varphi$ is in $C^{\alpha}(X)$ for some Hölder exponent
$\alpha'>0.$ By the complex generalization of Evans-Krylov theory
in \cite{way} it then follows that $\varphi_{\beta}$ is in $C^{2,\alpha}(X)$
for some $\alpha>0.$ Moreover, if $(\omega+dd^{c}\varphi)$ has coefficents
in $L^{\infty}$ then elliptic boot strapping gives that $\varphi_{\beta}$
is in $C^{3,\alpha}$ for any $\alpha<1$ and Prop \ref{prop:laplace estimate in khler case}
shows that $\omega+dd^{c}\varphi_{\beta}\leq C'\omega.$ 
\end{proof}
In particular, the transcendtal Bergman measure $e^{k(P_{\beta}\varphi-\varphi)}dV$
is uniformly bounded from above as long as $(\omega+dd^{c}\varphi)^{n}$
has an $L^{\infty}-$density. For the ordinary Bergman measure the
corresponding uniform bound was recently established in \cite{b-bern2},
under the stronger assumption that $(\omega+dd^{c}\varphi)$ has coefficents
in $L^{\infty}.$ The latter result was used in the proof, involving
local Bergman metric approximations, of Chen's conjecture concerning
the convexity of the K-energy along weak geodesics in the closure
of the space of Kähler metrics.
\begin{rem}
\label{rem:chinh }Inspired by the first preprint version of the present
paper on ArXiv it was shown in \cite{c-n} how to use a genaralization
of the transcental Bergman kernels introduced here, using Hessian
equations as a substitute for Monge-Ampère equations, in order to
establish the corresponding conjectural global regularization result
for $(\omega,m)-$ subharmonic functions (i.e. usc functions $u$
such that $(\omega+dd^{c}u)^{p}\wedge\omega^{n-p}\geq0$ for $p=1,2..,m;$
the case $m=n$ corresonds to the present setting). The elegant argument
in \cite{c-n} uses the notion of viscocity solutions of Hessian equations
based on the technique introduced in \cite{e-g-z}. One an important
observation in \cite{c-n} is that the convergence of $u_{\beta}$
towards $u_{\theta}$ (as in Prop \ref{prop:explicit rate in k=0000E4hler case})
implies the orthoganlity relation \ref{eq:orthog relation}, also
in the setting of Hessian equations, which forms the basis for the
variational approach to such equations developed in \cite{c-n}.
\end{rem}

\section{\label{sub:Degenerations-induced-by}Degenerations induced by a divisor }

Let now $(X,\omega)$ be a compact Kähler manifold with a fixed divisor
$Z,$ i.e. $Z$ is cut out by a holomorphic section $s$ of a line
bundle $L\rightarrow X.$ We identify the divisor $Z$ with the corresponding
current of integration $[Z]:=[s=0].$ Let us also fix a smooth Hermitian
metric $\left\Vert \cdot\right\Vert $ on $L$ and denote by $\theta_{L}$
its normalized curvature form. Fixing a parameter $\lambda\in[0,1[$
we set 
\begin{equation}
\varphi_{\lambda}:=\sup\{\varphi\,\,\,\varphi\leq0,\,\,\,\varphi\leq\lambda\log\left\Vert s\right\Vert ^{2}+O(1)\}\label{eq:def of envelope with prescribed lelong}
\end{equation}
The upper bound on $\varphi$ is equivalent to demanding that $\nu_{Z}(\varphi)\geq\lambda,$
where $\nu_{Z}(\varphi)$ denotes the Lelong number of $\varphi$
along $Z.$ To the pair $([\omega],Z)$ we associate the following
two constants: 

\[
\epsilon:=\sup\left\{ \lambda:\,[\omega]-\lambda[Z]\,\mbox{is\,\ensuremath{\mbox{Kähler}}}\right\} 
\]
and 
\[
\epsilon':=\sup\left\{ \lambda:\,[\omega]-\lambda[Z]\,\mbox{is\,\ensuremath{\mbox{big}}}\right\} ,
\]
 so that $\epsilon\leq\epsilon'$ (the constants $\epsilon$ and $\epsilon'$
appears as nef and psef thresholds, respectively, in the algebraic
geometry litterature). In the following we will always assume that
$\lambda\in[0,\epsilon'[,$ which ensures that $\varphi_{\lambda}$
is not identically equal to $-\infty.$ 

Set $u_{\lambda}:=\varphi_{\lambda}-\lambda\log\left\Vert s\right\Vert ^{2},$
defining a function in $PSH(X,\theta),$ where $\theta:=\omega-\lambda\theta_{L}.$
Equivalently, 
\begin{equation}
u_{\lambda}:=P_{\theta}(-\lambda\log\left\Vert s\right\Vert ^{2})\label{eq:envelope u lambda defined by project}
\end{equation}
 in the sense of formula \ref{eq:def of proj operator in khler case}.
This is equivalent to the construction of envelopes of metrics with
prescribed singularities out-lined in the introduction of \cite{berm1}
(see also \cite{r-w1} where it is shown that $u_{\lambda}$ is in
$\mathcal{C}_{loc}^{1,1}(X-Z))$ in the case of an integral class).

Note that it follows immediately from the definition that $u_{\lambda}$
has minimal singularities. In particular, if $\lambda<\epsilon,$
then $u_{\lambda}$ is bounded. In fact, $u_{\lambda}$ is even continuous.
The point is that, as long as the function $\varphi_{0}$ is lower
semi-continuous the corresponding envelope $P_{\theta}(\varphi_{0})$
will also be continuous. Indeed, it follows immediately that $P_{\theta}(\varphi_{0})^{*}\leq\varphi_{0}$
and hence $P_{\theta}(\varphi_{0})^{*}=P_{\theta}(\varphi_{0}),$
showing upper-semi continuity. The lower semi-continuity is then a
standard consequence of Demailly's approximation theorem applied to
the Kähler class $[\theta]$ (Theorem \ref{thm:Demailly}).
\begin{thm}
\label{thm:def along divisor}Let $(X,\omega)$ be a Kähler manifold
and $Z$ a divisor on $X$ and fix a positive number $\lambda<\epsilon'.$
Setting $\theta:=\omega-\lambda\theta_{L},$ let $u_{\beta,\lambda}$
be the unique $\theta-$psh function with minimal singularities solving
\[
(\theta+dd^{c}u)^{n}=e^{\beta u}\left\Vert s\right\Vert ^{2\lambda\beta}dV
\]
Then $u_{\beta,\lambda}$ converges uniformly, as $\beta\rightarrow\infty,$
to the envelope $u_{\lambda}.$ More precisely, 
\[
\sup_{X}\left|u_{\beta,\lambda}-u_{\lambda}\right|\leq\delta_{\beta}
\]
for some family of positive numbers $\delta_{\beta}$ (independent
of $\lambda)$ tending to $0$ as $\beta\rightarrow\infty.$ Moreover,
if $\lambda<\epsilon',$ then $\theta\text{+}dd^{c}u_{\beta,\lambda}\leq C\omega$
and hence the convergence holds in $C^{1,\alpha}(X)$ for any $\alpha<1.$\end{thm}
\begin{proof}
Set $f:=-\left\Vert s\right\Vert ^{2},$ which is a lsc function $X\rightarrow]-\infty,\infty]$
such that $dd^{c}f\leq C\omega.$ The convergence in energy and hence
the uniforme convergence then follows as before. Finally, the uniform
bound on $dd^{c}u_{\beta,\lambda}$ is obtained by writing $f$ is
a decreasing limit of smooth function $f_{j}$ such that $dd^{c}f_{j}\leq C'\omega,$
applying Proposition cr for a fixed $j$ and finally letting $j\rightarrow\infty.$ 
\end{proof}
Note that $\varphi_{\lambda,\beta}:=u_{\lambda}+\lambda\log\left\Vert s\right\Vert ^{2}\in PSH(X,\omega)$
is uniquely determined by the following equation on $X-Z:$ 
\begin{equation}
(\omega+dd^{c}\varphi_{\lambda,\beta})^{n}=e^{\beta\varphi_{\lambda,\beta}}dV\label{eq:equation for phi with asymptotics along Z}
\end{equation}
 together with the asymptotics $\varphi_{\lambda,\beta}=\lambda\log\left\Vert s\right\Vert ^{2}+O(1)$
close to $Z.$
\begin{rem}
More generally, it is enough to assume that $\omega$ is semi-positive
and big; then the uniform bound on $dd^{c}u_{\beta,\lambda}$ in the
previous theorem holds on any compact subset of the Kähler locus of
$X$ (by Prop \ref{prop:laplace in nef and big}). For example, this
situation appears naturally when $Z$ is the expectional divisor in
the blow-up of a point on a Kähler manifold $(M,\omega_{M})$ and
$\omega$ is the pull-back of $M.$ Then the corresponding constant
$\epsilon$ is the Seshadri constant of $p$ wrt $[\omega_{M}].$ 
\end{rem}

\section{\label{sec:Applications-to-geodesic}Applications to geodesic rays
and test configurations}

Let us start by briefly recalling the notions of geodesic rays and
test configurations in Kähler geometry (see \cite{p-s2,r-w3} and
references therein). Given an $n-$dimensional Kähler manifold $(X,\omega)$
we denote by $\mathcal{K}_{\omega}$ the space of all $\omega-$Kähler
potentials $\varphi,$ i.e. $\varphi$ is smooth and $\omega+dd^{c}\varphi>0$
(which equivalently means that $\varphi$ is in the interior of the
space $PSH(X,\omega)\cap C^{\infty}(X)).$ The infinite dimensional
space $\mathcal{K}_{\omega}$ comes with a canonical Riemannian metric,
the Mabuchi-Semmes-Donaldson metric. The corresponding geodesics rays
$\varphi^{t}(x)$ satisfy a PDE on $X\times[0,\infty[$ which, upon
complexification of $t$ (where $t:=-\log|\tau|^{2})$ is equivalent
to an $S^{1}-$invariant smooth solution to the Dirichlet problem
for the Monge-Ampère equation on the product $X\times\Delta^{*}$
of $X$ with the punctured unit-disc in the one-dimensional complex
torus $\C^{*}.$ In other words, $\varphi(x,\tau):=\varphi^{t}(x)$
satisfies 
\[
(dd^{c}\varphi+\pi^{*}\omega)^{n+1}=0,\,\,\,\,\mbox{on }X\times\Delta^{*}
\]
and $\varphi^{t}$ is called a\emph{ subgeodesic} if $dd^{c}\varphi+\pi^{*}\omega\geq0.$
In the case of an integral class $[\omega],$ i.e. when the class
is equal to the first Chern class $c_{1}(L)$ of a line bundle $L,$
there is a particularly important class of (weak) geodesics which
are associated to so called\emph{ test configurations} for $(X,L).$
This is an algebro-geometric gadget which gives an appropriate $\C^{*}-$equivariant
polarized closure $\mathcal{X}$ of $X\times\C^{*}$ over $\C.$ More
precisely, the data defining a test configuration $(\mathcal{X},\mathcal{L})$
for $(X,L)$ consists of 
\begin{itemize}
\item A normal variety $\mathcal{X}$ with a $\C^{*}-$action and flat equivariant
map $\pi:\mathcal{X}\rightarrow\C$ 
\item A relatively ample $\Q-$line bundle $\mathcal{L}$ over $\mathcal{X}$
equipped with an equivariant lift $\rho$ of the $\C^{*}-$action
on $X$
\item An isomorphism of $(X,L)$ with $(\mathcal{X},\mathcal{L})$ over
$1\in\C$
\end{itemize}
Here, we note that a ``transcendental'' analog of a test configuration
can be defined in the setting of non-integer classes. 
\begin{defn}
Let $(X,[\omega])$ be a complex manifold equipped with a Kähler class
$[\omega].$ A test configuration for $(X,[\omega])$ consists of
the following data:\end{defn}
\begin{itemize}
\item A normal Kähler space $\mathcal{X}$ equipped with a holomorphic $S^{1}-$action
and a flat holomorphic map $\pi:\mathcal{X}\rightarrow\C.$ 
\item An $S^{1}-$equivariant embedding of $X\times\ensuremath{\C^{*}}$
in $\mathcal{X}$ such that $\pi$ commutes with projection onto the
second factor of $X\times\ensuremath{\C^{*}}.$
\item A $(1,1)-$cohomology Kähler class $[\Omega]$ on $\mathcal{X}$ whose
restriction to $X\times\{1\}$ may be identified with $[\omega]$
under the previous embedding.
\end{itemize}
In particular, a test configuration $(\mathcal{X},\mathcal{L})$ for
a polarized variety $(X,L)$ induces a test configuration for $(X,c_{1}(L)).$
The point is that the $\C^{*}-$action on $(\mathcal{X},\mathcal{L})$
induces the required isomorphism between $\mathcal{X}$ and $X\times\C^{*}$
over $\C^{*}.$ 

Next, we explain how to obtain geodesic rays from a test configuration.
Given a test configuration $(\mathcal{X},[\Omega])$ for $(X,[\omega])$
we fix a smooth representative form $\Omega$ which is $S^{1}-$invariant.
For the sake of notational simplicity we also assume that $\Omega$
coincides with $\omega$ on $X\times\{1\}.$ First we let $\Phi$
be the unique bounded $\Omega-$psh function on $\mathcal{M}:=\pi^{-1}(\Delta)\subset\mathcal{X}$
satisfying the Dirichlet problem
\begin{equation}
(dd^{c}\Phi+\Omega)^{n+1}=0,\,\,\,\,\mbox{on\,\ensuremath{\mbox{int(}\mathcal{M})}}\label{eq:dirichlet problem on M}
\end{equation}
with vanishing boundary values (in the sense that $\Phi(p)\rightarrow0$
as $p$ approaches a point in $\partial\mathcal{M}$). In fact, it
can be shown, that $\Phi$ is automatically continuous up to the boundary
(see below). Next, we fix an $S^{1}-$invariant function $F$ on $X\times\C^{*}$
such that 
\[
\Omega=\pi^{*}\omega+dd^{c}F
\]
 and set $\varphi:=\Phi+F,$ which gives a correspondence 
\begin{equation}
PSH(X\times\C^{*},\Omega)\longleftrightarrow PSH(X\times\C^{*},\pi^{*}\omega),\,\,\,\,\Phi\leftrightarrow\varphi\label{eq:correspondence}
\end{equation}
Setting $\varphi^{t}(x):=\varphi(x,\tau)$ for $\varphi$ corresponding
to the solution $\Phi$ of the Dirichlet problem \ref{eq:dirichlet problem on M}
then defines the geodesic ray in question.

Let us also recall that the solution $\Phi$ of the Dirichlet problem
\ref{eq:dirichlet problem on M} may alternatively be defined as the
following envelope:
\begin{equation}
\mbox{\ensuremath{\Phi(x)}=\ensuremath{\sup}}\left\{ \Psi(x):\,\,\:\Psi\in PSH(\mathcal{M},\Omega):\,\,\,\Psi_{\partial\mathcal{M}}\leq0\right\} \label{eq:dirchlet solution as envelope}
\end{equation}
As shown in \cite{r-w3}, in the line bundle case, the geodesic ray
$\varphi^{t}$ may be realized as a Legendre transform of certain
envelopes determined by the test configuration. Here we note that
the latter result may be generalized to the ``transcendental'' setting.
To this end first observe that a test configuration $(\mathcal{X},[\Omega])$
for $(X,[\omega])$ determines a concave decreasing family 
\[
\mathcal{F}^{\mu}(X,\omega)\subset PSH(X,\omega)
\]
of convex subspaces indexed by $\mu\in\R,$ defined as follows: the
subspace $\mathcal{F}^{\mu}(X,\omega)$ consists of all $\varphi$
in $PSH(X,\omega)$ such that, setting $\bar{\varphi}(x,t):=\varphi(x),$
the current 
\[
dd^{c}(\bar{\varphi}-\mu\log|\tau|^{2})+\pi^{*}\omega
\]
 on $X\times\C^{*}$ extends to a positive current on $\mathcal{X}$
in $[\Omega].$ In other words, we demand that the current $dd^{c}\bar{\varphi}+\pi^{*}\omega$
extends to current on $\mathcal{X}$ in $[\Omega]$ with Lelong number
at least $\mu$ along the central fiber of $\mathcal{X}$ (in a generalized
sense, as we are allowing negative Lelong numbers). The family $\mathcal{F}^{\mu}(X,\omega),$
thus defined, is clearly a concave decreasing family of convex subspaces
(it is the ``psh analogue'' of the filtrations of $H^{0}(X,kL)$
defined in \cite{witt,r-w3}). Next, to the family $\mathcal{F}^{\mu}(X,\omega)$
we associate the following family of envelopes $\psi_{\mu}$ in $PSH(X,\omega):$
\begin{equation}
\psi_{\mu}(x):=\sup_{\psi\in\mathcal{F}^{\mu}(X,\omega)}\left\{ \psi(x),\,\,\,\psi\leq0\right\} ,\label{eq:def of envelope wrt concave family}
\end{equation}

\begin{prop}
\label{prop:dirichlet as legendre}Let $(\mathcal{X},[\Omega])$ be
a test configuration for $(X,[\omega]).$ Then the corresponding geodesic
ray $\varphi^{t}$ in $PSH(X,\omega)$ may be realized as the Legendre
transform (wrt $t)$ of the envelopes $\psi_{\mu},$ i.e. 
\[
\varphi^{t}(x)=\sup_{\mu\in\R}\left\{ \psi_{\mu}(x)+\mu t\right\} 
\]
\end{prop}
\begin{proof}
By the definition of the envelopes it is equivalent to prove that
\[
\varphi^{t}(x)=\sup_{\psi_{\mu}}\left\{ \psi_{\mu}(x)+\mu t\right\} 
\]
where the sup ranges over\emph{ all} $\psi_{\mu}\in\mathcal{F}^{\mu}(X,\omega)$
with $\psi_{\mu}\leq0$ on $X.$ Using the correspondence \ref{eq:correspondence}
we may identify $\psi_{\mu}(x)+\mu t$ with a function $\Phi_{\mu}$
in $PSH(X\times\C^{*},\Omega),$ which, by the extension assumption
for the elements in the subspace $\mathcal{F}^{\mu}(X,\omega),$ extends
uniquely to define an element in $PSH(\mathcal{X},\Omega)$ (which
by construction vanishes on the boundary of $\mathcal{M}).$ But then
$\Phi_{\mu}\leq\Phi,$ the envelope defining the geodesic ray $\varphi^{t}.$
This proves the lower bound on $\varphi^{t}(x).$ To prove the upper
bound we note that, by the convexity in $t,$ we may write 
\[
\varphi^{t}(x)=\sup_{\mu\in\R}\left\{ \phi_{\mu}^{*}(x)+\mu t\right\} ,
\]
 where $\phi_{\mu}^{*}$ is the Legendre transform, wrt $t,$ of $\varphi^{t}$
(with our sign conventions $\phi_{\mu}^{*}$ is thus concave wrt $\mu):$
\[
\phi_{\mu}^{*}(x)=\inf_{t}\left\{ \mu t+\varphi^{t}(x)\right\} 
\]
 In particular, $\phi_{\mu}^{*}(x)+\mu t\leq\varphi^{t}$ and moreover,
by Kiselman's minimum principle, $\phi_{\mu}^{*}(x)$ is $\omega-$psh
on $X.$ Identifying $\phi_{\mu}^{*}(x)+\mu t$ with a function $\Phi_{\mu}$
in $PSH(X\times\C,\Omega),$ as before, it thus follows that $\Phi_{\mu}\leq\Phi.$
In particular, $\Phi_{\mu}$ is bounded from above and thus extends
to define an element in $PSH(\mathcal{X},\Omega),$ i.e. the corresponding
curvature current is positive. But this means that $\phi_{\mu}^{*}(x)\in\mathcal{F}^{\mu}(X,\omega)$
which concludes the proof of the upper bound.\end{proof}
\begin{example}
\label{(deformation-to-the}(deformation to the normal cone; compare
\cite{r-t,r-t0}). Any given (say reduced) divisor $Z$ in $X$ determines
a special test configuration whose total space $\mathcal{X}$ is the
deformation to the normal cone of $Z.$ In other words, $\mathcal{X}$
is the blow-up of $X\times\C$ along the subscheme $Z\times\{0\}.$
Denote by $\pi$ the corresponding flat morphism $\mathcal{X}\rightarrow\C$
which factors through the blow-down map $p$ from $\mathcal{X}$ to
$X\times\C.$ This construction also induces a natural embedding of
$X\times\C^{*}$ in $\mathcal{X}.$ Given a Kähler class $[\omega]$
on $X,$ which we may identify with a class on $X\times\C$ and a
positive number $c$ we denote by $[\Omega_{c}]$ the corresponding
class $[p^{*}\omega]-c[E]$ on $\mathcal{X},$ where $E$ is the exceptional
divisor and we are assuming that $c<\epsilon,$ where $\epsilon$
is defined as the sup over all positive numbers $c$ such that the
class $[\Omega_{c}]$ is Kähler (i.e. $\epsilon$ is the Seshadri
constant of $Z$ wrt $[\omega]).$ In this setting it is not hard
to check that $\varphi\in\mathcal{F}^{\mu}(X,\omega)$ iff $\nu_{Z}(\varphi)\geq\mu+c,$
where $\nu_{Z}(\varphi)$ denotes the Lelong number of $\varphi$
along the divisor $Z$ in $X.$ The point is that $[p^{*}\omega]-cE$
may be identified with the subspace of currents in $[p^{*}\omega]$
with Lelong number at least $c$ along the divisor $E$ in $\mathcal{X}$
which in this case is equivalent to having Lelong number at least
$c$ along the central fiber $[\mathcal{X}_{0}],$ which in turn is
equivalent to $\varphi$ having Lelong number at least $c$ along
$Z$ in $X.$ In particular, setting $\mu=\lambda-c$ we have $\varphi_{\lambda}=\psi_{\mu},$
where $\varphi_{\lambda}$ is the envelope defined by formula \ref{eq:def of envelope with prescribed lelong},
i.e. $u_{\lambda}=\psi_{\mu}-\lambda\log\left\Vert s\right\Vert ^{2},$
where $u_{\lambda}$ is defined by \ref{eq:envelope u lambda defined by project}. 
\end{example}
Now we observe that one obtains a family of subgeodesics, approximating
the weak geodesic $\varphi^{t}$ in the closure of $\mathcal{K}_{\omega},$
associated to a divisor $Z$ and a number $c\in[0,\epsilon[,$ as
in the previous example, by setting 
\[
\varphi_{\beta}^{t}:=\frac{1}{\beta}\log\int_{[0,c]}d\lambda e^{\beta\left((\lambda-c)t+\varphi_{\lambda,\beta}\right)},
\]
 where $\varphi_{\lambda,\beta}$ is the regularization of $\varphi_{\lambda}$
introduced in Section \ref{sub:Degenerations-induced-by}, solving
the Monge-Ampère equation \ref{eq:equation for phi with asymptotics along Z}
(which is indeed a subgeodesic as it is a superposition of the subgeodesics
$(\lambda-c)t+\varphi_{\lambda,\beta}$). Combining Theorem \ref{thm:def along divisor}
with the previous proposition we arrive at the following
\begin{thm}
\label{cor:geod ray in text}Let $[\omega]$ be a Kähler class on
$X$ and $Z$ a divisor in $X$ and fix a positive number $c\in[0,\epsilon[.$
Then the corresponding subgeodesics $\varphi_{\beta}^{t}$ converge,
as $\beta\rightarrow\infty,$ to the weak geodesic $\varphi^{t},$
uniformly on $X\times[0,T[$ for any fixed $T<\infty$ (and for $T=\infty$
in the case when $[\omega]\in H^{2}(X,\Q)).$ Moreover, the first
order space-time derivatives of $\varphi_{\beta}^{t}$ are uniformly
bounded on $X\times[0,\infty[.$ \end{thm}
\begin{proof}
By Theorem \ref{thm:def along divisor} 
\[
\varphi_{\beta}^{t}=\frac{1}{\beta}\log\int_{[0,\Lambda]}d\lambda e^{\beta((\lambda-c)t+\varphi_{\lambda})}+o(1),\,\,\,\varphi_{\lambda}:=u_{\theta,\lambda}+\lambda\log\left\Vert s\right\Vert ^{2},
\]
 where the $o(1)-$term is independent of $t$ and converges uniformly
to $0$ on $X\times[0,c]$ as $\beta\rightarrow\infty.$ As a consequence,
for $t\in[0,T]$ we clearly have 
\[
\varphi_{\beta}^{t}=\sup_{\mu\in[-c,0]}\left(\mu t+\psi_{\mu}\right)+o(1)
\]
(where, as explained in the previous example, $\psi_{\mu}=\varphi_{\lambda}$
for $\mu=\lambda-c)$ and by Prop \ref{prop:dirichlet as legendre}
the first term above defines the desired geodesic ray $\varphi^{t}.$
Finally, we need to show that the error term above is uniform at $T\rightarrow\infty$
in the case when $[\omega]\in H^{2}(X,\Q)).$ To this end we will
use a compactification argument. Set, as before $t=-\log|\tau|^{2},$
where $\tau\in\C^{*}.$ By the definition of the deformation to the
normal cone $\mathcal{X}$ (see the previous example) the function
$\Phi_{\mu}$ defined in the proof of Prop \ref{prop:dirichlet as legendre}
defines an $\Omega$-psh function on $\mathcal{X}.$ We thus a get
a family of functions on $\mathcal{X}$ defined by 
\[
\Psi_{\beta}:=\frac{1}{\beta}\log\int_{[-c,0]}d\mu e^{\beta\Phi_{\mu}}
\]
 and such that $\Psi_{\beta}$ increases (by Hölder's inequality)
to the function $\Psi_{\infty}:=\sup_{\mu}\Phi_{\mu},$ which, according
to the proof of Prop \ref{prop:dirichlet as legendre}, coincides
with the envelope $\Phi$ defined by formula \ref{eq:dirchlet solution as envelope}.
But the latter envelope is continuous (up to the boundary) on $\mathcal{M}$
and hence it follows from Dini's lemma that $\Psi_{\beta}$ converges
to $\Psi$ uniformly, as desired. The continuity of the envelope $\Phi$
follows from standard arguments in the case when $\mathcal{M}$ is
smooth and the back-ground form $\eta$ is Kähler. We recall that
the argument just uses that any sequence of $\eta-$psh functions
may be approximated by a decreasing sequence of continuous $\eta-$psh
functions, as follows from the approximation results in \cite{dem0}
(see for example \cite{b-d} for a similar situation). The latter
approximation property has been generalized, in the case of rational
classes, to the case when $\eta$ is merely assumed to be semi-positive
(and big) \cite{c-g-z} and hence the proof of the continuity still
applies in the present situation (strictly speaking the results in
op. cit. apply to \emph{compac}t complex manifolds, but we can simply
pass to a resolution of the the $\C^{*}-$equivariant compactification
of $\mathcal{X}$ fibered over the standard $\P^{1}-$compactification
of $\C$ and adopt the argument using barriers in \cite{berman6ii}).

Finally, to prove the last statement we observe that, fixing a first
order differential operator $D_{x}$ on $X,$ we have 
\[
\frac{d}{dt}\varphi_{\beta}^{t}(x):=\int_{[0,c]}(\lambda-c)\nu_{(x,t)}^{(\beta)}(\lambda),\,\,\,\,\,D_{x}\varphi_{\beta}^{t}(x)=\int_{[0,c]}D_{x}\varphi_{\beta,\lambda}(x)\nu_{(x,t)}^{(\beta)}(\lambda),
\]
 in terms of the following probability measure $\nu_{(x,t)}^{(\beta)}$
on $[0,c]:$ 
\[
\nu_{(x,t)}^{(\beta)}(\lambda):=e^{\beta\left((\lambda-c)t+\varphi_{\lambda,\beta}\right)}/\int_{[0,c]}d\lambda e^{\beta\left((\lambda-c)t+\varphi_{\lambda,\beta}\right)}
\]
But then the estimate on the time derivative follows immediately from
the uniform bound $|\lambda|\leq c$ and the estimate on the space
derivative form the uniform bound on $D_{x}\varphi_{\beta,\lambda}$
(Theorem \ref{thm:def along divisor}).\end{proof}
\begin{rem}
In the case when $[\omega]=c_{1}(L)$ it was shown in\cite{p-s-4}
how to approximate (in a point-wise almost everywhere sense) a weak
geodesic $\varphi_{t}$ associated to a test configuration by smooth
Bergman geodesics associated to higher powers of the line bundle $L$
(see also \cite{r-w3} for an alternative proof). Accordingly, it
seems natural to view $\varphi_{\beta}^{t}$ as a transcendtal analog
of the Phong-Sturm Bergman geodesics. One advantage of $\varphi_{\beta}^{t}$
is that the convergence is uniform (even when $t$ is not constrained
to be in a bounded interval in the case of a rational class). Assuming
the conjectural validity of the appoximation result in \cite{c-g-z}
for general transcendental classes, the uniformity in the previous
theorem holds for $T=\infty,$ in general. It is also interesting
to compare the bound on the first derivatives above with the case
of toric Bergman geodesics studied in \cite{s-z}, where uniform $C^{1}-$convergence
is established. It seems likely that a similar $C^{1}-$convergence
holds in the present setting (even in the general non-toric setting),
but we will not go further into this here. It would also be interesting
to see if there is a uniform bound on the space Laplacians of $\varphi_{\beta}^{t}$
(say on any fixed time inverval). 
\end{rem}

\subsubsection{General (analytic) test configurations}

Of course, the test configurations defined by the deformation to the
normal cone of a divisor are very special ones. But the convergence
result in Cor \ref{cor:geod ray in text} can be extended to general
test configurations for a polarized manifold $(X,L)$ (by replacing
$MA(u_{\beta,\lambda})$ with $MA(\varphi_{\beta,\mu})$ where $\varphi_{\beta,\mu}\in\mathcal{F}^{\lambda}(X,\omega)$
satisfies the equation \ref{eq:equation for phi with asymptotics along Z}).
The argument uses Odaka's generalization of the Ross-Thomas slope
theory \cite{o} defined in terms of a flag of ideals on $X.$ The
point is that by blowing up the corresponding ideals one sees that
the pullback of the corresponding envelopes $\psi^{\mu}$ have divisorial
singularities (compare Prop 3.22 in \cite{hi}) so that the previous
convergence argument can be repeated (as they apply also when $L$
is merely semi-ample and big, which is the case on the blow-up). 

More generally, an analytic generalization of test configurations
for a polarization $(X,L)$ was introduced in \cite{r-w3}. Similarly,
an\emph{ analytic test configuration }for a Kähler manifold $(X,\omega)$
may be defined as a concave family $[\psi^{\mu}]$ of singularity
classes in $PSH(X,\omega).$ The corresponding space $\mathcal{F}^{\mu}(X,\omega)$
may then be defined as all elements $\psi$ in such that $[\psi]=[\psi^{\mu}].$
To any such family one associates a family of envelopes $\psi_{\mu}$
defined by formula \ref{eq:def of envelope wrt concave family}. As
shown in \cite{r-w3} taking the Legendre transform of $\psi_{\mu}$
wrt $\mu$ gives a curve $\varphi^{t}$ in $PSH(X,\omega)$ which
is a weak geodesic. The regularization scheme introduced in this paper
could be adapted to this general framework by first introducing suitable
algebraic regularizations of the singularity classes and using blow-ups
(as in \cite{o}). But we leave these developments and their relation
to K-stability and the Yau-Tian-Donaldson conjecture for the future.
For the moment we just observe that the latter conjecture admits a
natural generalization to transcendental classes.
\begin{example}
\label{ex:big test}Continuing with the previous example of deformation
to the normal cone, we observe that one obtains a (transcendtal) analytic
test configuration, which is not a bona fide test configuration, when
$c\in]\epsilon,\epsilon'[.$ In geometric terms this corresponds to
allowing the line bundle \emph{$\mathcal{L}$ }(or the corresponding
Kähler class on the total space) to be merely big. In this setting
the $C^{0}-$ convergence in Theorem \ref{cor:geod ray in text} still
holds (with the same proof) as long as $t$ is restricted to a bounded
interval. 
\end{example}

\subsubsection{\label{sub:A-generalization-of}A generalization of the Yau-Tian-Donaldson
conjecture to transcendetal classes.}

Using Wang's intersection formula \cite{wa} there is a natural generalization
of the notion of K-stability of a polarization $(X,L)$: by definition,
a Kähler class $[\omega]$ on $X$ is\emph{ K-stable} if, for any
test configuration $(\mathcal{X},[\Omega])$ for $(X,[\omega])$ the
corresponding Donaldson-Futaki invariant satisfies $DF(\mathcal{X},[\Omega])\geq0$
with equality iff $\mathcal{X}$ is equivariantly isomorphic to a
product. Similarly, K-polystability is defined by not requiring that
the isomorphism be equivariant. Here $DF(\mathcal{X},[\Omega])$ is
defined as the following sum of intersection numbers

\[
DF(\mathcal{X},[\Omega]):=a[\Omega]^{n+1}+(n+1)K_{\mathcal{X}/\P^{1}}\cdot[\Omega]^{n},\,\,\,\,a:=n(-K_{X})\cdot[\omega]^{n-1}/[\omega]^{n}
\]
 where we have replaced $\mathcal{X}$ with its equivariant compactification
over $\P^{1}$ and $[\Omega]$ with the corresponding class on the
compactification and the intersection numbers are computed on the
compactification. The transcendental version of the Yau-Tian-Donaldson
conjecture may then be formulated as the conjecture that $[\omega]$
admits a constant scalar curvature metric iff $(X,[\omega])$ is K-polystable.
It is interesting to compare this generalization with Demailly-Paun's
generalization of the Nakai-Moishezon criterium for ample line bundles
\cite{d-p}, which in the case when $X$ is a projective manifold
says that if a $(1,1)-$ class $[\theta]$ has positive intersections
with all $p-$dimensional subvarieties of $X$ then $[\theta]$ contains
a Kähler form $\omega.$ The difference is thus that in order to draw
the considerably stronger conclusion that $\omega$ can be chosen
to have constant scalar curvature one needs to impose conditions on
``secondary'' intersection numbers as well, i.e. intersection numbers
defined over all suitable degenerations of $(X,[\theta]).$ Finally,
it should be pointed out that it may very well be that the notion
of (transcendental) test configuration above has to be generalized
a bit further in order for the previous conjecture to stand a chance
of being true (compare the discussion in the introduction of the paper).

\end{document}